\documentclass[titlepage, 12pt]{article}
\usepackage[margin=2.5cm]{geometry}
\usepackage{mathtools}

\usepackage{lmodern}
\usepackage{csquotes}
\usepackage{enumitem}
\MakeOuterQuote{"}
\def\hyph{-\penalty0\hskip0pt\relax}

\usepackage[backend=biber,style=numeric, 
maxbibnames=20, 
url=false, 
isbn=false, 
doi=false]{biblatex}
\DeclareLabelalphaTemplate{
  \labelelement{
    \field[final]{shorthand}
    \field{label}
    \field[strwidth=3,strside=left,ifnames=1]{labelname}
    \field[strwidth=1,strside=left]{labelname}
  }
  \labelelement{
    \field{year}    
  }
}

\newbibmacro{string+doiurlisbn}[1]{%
  \iffieldundef{doi}{%
    \iffieldundef{url}{%
      \iffieldundef{isbn}{%
        \iffieldundef{issn}{%
          #1%
        }{%
          \href{http://books.google.com/books?vid=ISSN\thefield{issn}}{#1}%
        }%
      }{%
        \href{http://books.google.com/books?vid=ISBN\thefield{isbn}}{#1}%
      }%
    }{%
      \href{\thefield{url}}{#1}%
    }%
  }{%
    \href{http://dx.doi.org/\thefield{doi}}{#1}%
  }%
}

\NewBibliographyString{habilthesis}

\DeclareStyleSourcemap{
  \maps[datatype=bibtex,overwrite=false]{
    \map{
      \step[typesource=habilthesis, typetarget=thesis, final]
      \step[fieldset=type,fieldvalue=Habilitation Thesis]
    }
    \map{
      \pertype{thesis}
      \step[fieldset=type, fieldvalue=thesis]
    }
  }
}

\DeclareFieldFormat{title}{\usebibmacro{string+doiurlisbn}{\mkbibemph{#1}}}
\DeclareFieldFormat[article,incollection,inproceedings,phdthesis,book]{title}%
{\usebibmacro{string+doiurlisbn}{\mkbibquote{#1}}}

\DeclareFieldFormat{title}{%
  \usebibmacro{string+doiurlisbn}{\mkbibemph{#1}}}
\DeclareFieldFormat
  [article,inbook,incollection,inproceedings,patent,thesis,unpublished]
  {title}{\usebibmacro{string+doiurlisbn}{\mkbibquote{#1\isdot}}}
\DeclareFieldFormat
  [suppbook,suppcollection,suppperiodical]
  {title}{\usebibmacro{string+doiurlisbn}{#1}}

\usepackage{authblk}

\usepackage[hyphens]{xurl}
\usepackage[usenames, dvipsnames]{xcolor}
\definecolor{linkB}{RGB}{4, 106, 143}
\definecolor{defRED}{RGB}{84, 4, 18}
\usepackage[colorlinks=true, linkcolor={linkB}, urlcolor={linkB}, citecolor={linkB}]{hyperref}
\newcommand{\defin}[1]{\textcolor{defRED}{\emph{#1}}\index{#1}}
\usepackage[english]{babel}
\addto\extrasenglish{}
\addto\extrasenglish{}

\usepackage{amsmath, amssymb, amsthm}
\usepackage{thmtools}
\usepackage{thm-restate}

\declaretheorem[name=Lemma, numberwithin = section]{lemma}
\declaretheorem[name=Theorem,sibling = lemma]{theorem}
\declaretheorem[name=Proposition, sibling=lemma]{proposition}
\declaretheorem[name=Observation, sibling=lemma]{observation}

\declaretheorem[name=Corollary, sibling=lemma]{corollary}
\declaretheorem[name=Conjecture, sibling=lemma]{conjecture}

\declaretheorem[name=Claim]{claim}

\declaretheorem[name=Question]{question}
\declaretheorem[name=Problem]{problem}

\newenvironment{subproof}[1][\proofname]{%
  \begin{proof}[#1]%
}{%
  \end{proof}%
}

\usepackage[noabbrev,capitalise,nameinlink]{cleveref}
\crefname{theorem}{Theorem}{Theorems}
\crefname{proposition}{Proposition}{Propositions}
\crefname{lemma}{Lemma}{Lemmas}
\crefname{claim}{Claim}{Claims}
\crefname{subclaim}{Sub-Claim}{Sub-Claims}
\crefname{observation}{Observation}{Observations}
\crefname{remark}{Remark}{Remarks}
\crefname{corollary}{Corollary}{Corollaries}
\crefname{definition}{Definition}{Definitions}
\crefname{conjecture}{Conjecture}{Conjectures}
\crefname{question}{Question}{Questions}
\crefformat{equation}{(#2#1#3)}
\Crefformat{equation}{Equation #2(#1)#3}

\usepackage{chngcntr}

\DeclareMathOperator*{\argmin}{arg\,min}
\DeclareMathOperator*{\dom}{\operatorname{dom}}
\usepackage{pdfpages}
\usepackage{stackengine}
\stackMath

\newcommand{\gcap}{\hspace{-0.02cm}
    \mathbin{\vcenter{\offinterlineskip
    \hbox{\(\raisebox{0.025cm}{\scalebox{1.25}{$\cap$}}\)}\vskip-0.21cm\hbox{\hspace{0.1cm}\scriptsize\(\bullet\)}}}%
}

\newcommand{\gcup}{\hspace{-0.02cm}
    \mathbin{\vcenter{\offinterlineskip
    \hbox{\(\raisebox{0.025cm}{\scalebox{1.25}{$\cup$}}\)}\vskip-0.25cm\hbox{\hspace{0.1cm}\scriptsize\(\bullet\)}}}%
}

\newcommand{\g}{intersectionwise $\chi$-guarding}
\newcommand{\sg}{intersectionwise self-$\chi$-guarding}
\newcommand{\chib}{\chi\text{-bounded}}

\usepackage[classfont=sanserif, langfont=roman,funcfont=italic]{complexity}

\title{\LARGE{Intersections of graphs and $\chi$-boundedness}}

\author[$^{\mathsection}$]{Aristotelis Chaniotis}
\author[$^{\mathsection}$]{Hidde Koerts}
\author[$^{\mathsection}$,\thanks{We acknowledge the support of the Natural Sciences and Engineering Research Council of Canada (NSERC), [funding reference number RGPIN-2020-03912]. Cette recherche a \'et\'e financ\'ee par le Conseil de recherches en sciences naturelles et en g\'enie du Canada (CRSNG), [num\'ero de r\'ef\'erence RGPIN-2020-03912]. This project was funded in part by the Government of Ontario. This research was conducted while Spirkl was an Alfred P. Sloan Fellow.}]{Sophie Spirkl}

\affil[$^{\mathsection}$]{Department of Combinatorics and Optimization, University of Waterloo, \protect\\ Waterloo, Ontario, N2L 3G1, Canada \protect \vspace{0.25cm}}

\date{\today}

\begin{document}

\maketitle

\fontsize{12}{15}\selectfont

\begin{abstract}
   Given $k$ graphs $G_{1}, \ldots, G_{k}$, their {\em intersection} is the graph $(\cap_{i\in [k]}V(G_{i}), \cap_{i\in [k]}E(G_{i}))$. Given $k$ graph classes $\mathcal{G}_{1}, \ldots , \mathcal{G}_{k}$, we call the class $\{G: \forall i \in[k], \exists G_{i} \in \mathcal{G}_{i} \text{ such that } G=G_{1}\cap \ldots \cap G_{k}\}$ the {\em graph-intersection} of $\mathcal{G}_{1}, \ldots , \mathcal{G}_{k}$. The main motivation for the work presented in this paper is to try to understand under which conditions graph-intersection preserves $\chi$-boundedness. We consider the following two questions:

   \begin{itemize}
       \item 
    
    {\em Which graph classes have the property that their graph-intersection with every $\chi$-bounded class of graphs is $\chi$-bounded?} We call such a class {\em intersectionwise $\chi$-guarding}. We prove that classes of graphs which admit a certain kind of decomposition are \g{}. We provide necessary conditions that a finite set of graphs $\mathcal{H}$ should satisfy if the class of $\mathcal{H}$-free graphs is \g{}, and we characterize the \g{} classes which are defined by a single forbidden induced subgraph.

\item
    {\em Which graph classes have the property that, for every positive integer $k$, their $k$-fold graph-intersection is $\chi$-bounded?} We call such a class {\em \sg}. We study \sg{} classes which are defined by a single forbidden induced subgraph, and we prove a result which allows us construct \sg{} classes from known \g{} classes.
   \end{itemize}
\end{abstract}

\newpage
\section{Introduction}
\label{sec:intro}

\subsection{Basic notation and terminology}

For terminology and notation not defined here we refer readers to \cite{west2020combinatorial}. Unless otherwise stated, graphs in this paper are finite, undirected, and have no loops or parallel edges. Graph classes in this paper are assumed to be \defin{hereditary}, that is, closed under isomorphism and under taking induced subgraphs. Let $G$ be a graph. We denote the complement of $G$ by $G^{c}$. In this paper we often denote an edge $\{u,v\}\in E(G)$ by $uv$. Let $X\subseteq V(G)$. We denote the subgraph of $G$ which is induced by $X$, by $G[X]$. Let $H$ be a graph. We say that $G$ is \defin{$H$-free} (respectively \defin{contains $H$}) if it contains no (respectively contains an) induced subgraph isomorphic to $H$. Let $\mathcal{H}$ be a set of graphs. We say that $G$ is $\mathcal{H}$-free if $G$ is $H$-free for every $H\in \mathcal{H}$. We denote by $\{\mathcal{H}\text{-free graphs}\}$ the class of $\mathcal{H}$-free graphs. The \defin{neighborhood of $X$ in $G$}, denoted by $N_{G}(X)$, is the set $\{u\in V(G): \exists v\in X, uv\in E(G)\}$. The \defin{closed neighborhood of $X$ in $G$}, denoted by $N_{G}[X]$, is the set $N_{G}(X)\cup \{X\}$. When $X=\{u\}$ we write $N_{G}(u)$ and $N_{G}[u]$ instead of $N_{G}(\{u\})$ and $N_{G}[\{u\}]$. We use $A_{G}(X)$ to denote the set $V(G)\setminus N_{G}[X]$. When there is no danger of ambiguity, we omit the subscripts from the notations of neighborhoods and from $A_{G}(X)$. For a positive integer $t$ we denote by $K_{t}, P_{t}$, and $C_{t}$ a complete graph, a path, and a cycle on $t$ vertices, respectively.

Let $k$ and $t$ be positive integers. We denote the set $\{1,\ldots,k\}$ by $[k]$. We also denote  by $[k]^{t}$ the set of $t$-tuples with elements from $[k]$. A \defin{$k$-coloring} of $G$ is a function $f:V(G)\to [k]$. A \defin{coloring} of $G$ is a function which is a $k$-coloring of $G$ for some $k$. Let $\mathcal{P}$ be a class of graphs. We say that a $k$-coloring $f$ of $G$ is a \defin{$\mathcal{P}$ $k$-coloring} of $G$ if for every $i\in[k]$ we have that $G[f^{-1}(i)] \in \mathcal{P}$. The \defin{$\mathcal{P}$ chromatic number} of $G$ is the minimum integer $k$ for which $G$ admits a $\mathcal{P}$ $k$-coloring. If $\mathcal{P}$ is the class of edgeless graphs, then we call a $\mathcal{P}$ $k$-coloring a \defin{proper $k$-coloring}, and we refer to the $\mathcal{P}$ chromatic number of a graph $G$ simply as the \defin{chromatic number} of $G$. We denote the chromatic number of $G$ by $\chi(G)$.
 
The \defin{clique number} (respectively the \defin{independence number}) of $G$, denoted by $\omega(G)$ (respectively by $\alpha(G)$), is the size of a largest clique (respectively independent set) of $G$. It is immediate that for every graph $G$, we have $\omega(G) \leq \chi(G)$. A \defin{triangle} in $G$ is a subgraph induced by a clique of size three. Tutte (under the alias Blanche Descartes) \cite{descartes1947three, descartes1954solution} and Zykov \cite{zykov1949some} independently proved in the late 1940's that graphs of large chromatic number do not necessarily contain large complete subgraphs; in particular, they proved that there exist triangle-free graphs of arbitrarily large chromatic number, and thus the gap between $\omega$ and $\chi$ can be arbitrarily large. 

\subsection{\texorpdfstring{$\chi$}{TEXT}-boundedness and graph operations}

Following Gy\'arf\'as \cite{gyarfas1985problems}, we say that a class of graphs $\mathcal{C}$ is \defin{$\chi$-bounded} if there exists a non-decreasing function $f:\mathbb{N} \to \mathbb{N}$ such that for every $G\in \mathcal{C}$, we have that $\chi(G) \leq f(\omega(G))$; in this case $f$ is called a \defin{$\chi$-bounding function} for $\mathcal{C}$. If $f$ can be chosen to be polynomial, then $\mathcal{C}$ is \defin{polynomially $\chi$-bounded}. A graph $G$ is called \defin{perfect} if $\chi(H) = \omega(H)$ for every induced subgraph $H$ of $G$. Thus the class of perfect graphs is the largest class of graphs which is  $\chib$ by the identity function. 

{\em Which classes of graphs are $\chib$?} Since the class of all graphs is not $\chib$, it follows that every $\chib$ class is defined by a set of forbidden induced subgraphs. Despite great progress in the study of $\chi$-boundedness in the last years (see \cite{scott2022graphs, scott2020survey} for recent surveys), we are still far from a good understanding of $\chib$ classes of graphs. 

One approach that can be and has been used in order to prove that a class of graphs $\mathcal{C}$ is $\chib$ is the following: {\em We prove that every graph in $\mathcal{C}$ can be constructed by applying an operation which preserves $\chi$-boundedness to graphs that are members of $\chib$ classes}. For example the celebrated Strong Perfect Graph Theorem \cite{chudnovsky2006strong} was proved using this approach.

{\em Which operations preserve $\chi$-boundedness?} Among operations that are known to preserve $\chi$-boundedness are the following: gluing along a clique, gluing along a bounded number of vertices \cite{alon1987subgraphs} (see also \cite{girao2022subgraphs}), substitution \cite{chudnovsky2013substitution}, $1$-joins \cite{bonamy2020graphs, dvovrak2012classes}, and amalgams \cite{penev2017amalgams}. In the case that an interesting operation does not preserve $\chi$-boundedness in general, it is reasonable to try to characterize the cases where it does. 

In this paper, we consider the operation of the graph-intersection between classes of graphs, which arises from the operation of intersection among a finite number of graphs (precise definitions later on in this section). To the best of our knowledge, Gy\'arf\'as \cite[Section 5]{gyarfas1985problems} first considered the interplay of this operation with $\chi$-boundedness. In recent work, Adenwalla, Braunfeld, Sylvester, and Zamaraev \cite{adenwalla2024boolean} considered this topic in the context of their broader study on boolean combinations of graphs. Since, as we will see later on this section, graph-intersection does not preserve $\chi$-boundedness in general, we focus on understanding under which conditions it does.

\subsection{Terminology and preliminaries on graph-intersection and \texorpdfstring{$\chi$}{TEXT}-boundedness}

Let $k$ be a positive integer. Given $k$ graphs $G_{1}, \ldots, G_{k}$, their intersection (respectively union) is the graph $\left(\cap_{i\in[k]}V(G_{i}), \cap_{i\in[k]}E(G_{i})\right)$ (respectively $\left(\cup_{i\in[k]}V(G_{i}), \cup_{i\in[k]}E(G_{i})\right)$). Given $k$ graph classes $\mathcal{G}_{1}, \ldots , \mathcal{G}_{k}$, we call the class $\{G: \forall i \in[k], \exists G_{i} \in \mathcal{G}_{i} \text{ such that } G=G_{1}\cap \ldots \cap G_{k}\}$ (respectively the class $\{G: \forall i \in[k], \exists G_{i} \in \mathcal{G}_{i} \text{ such that } G=G_{1}\cup \ldots \cup G_{k}\}$) the \defin{graph-intersection} (respectively \defin{graph-union}) of $\mathcal{G}_{1}, \ldots , \mathcal{G}_{k}$, and we denote this class by $\mathcal{G}_{1} \gcap \ldots \gcap \mathcal{G}_{k}$ (respectively $\mathcal{G}_{1} \gcup \ldots \gcup \mathcal{G}_{k}$). Given the definitions of these operations and the discussion above about operations which preserve $\chi$-boundedness, it is natural to ask whether or not these operations preserve $\chi$-boundedness.

We first consider the graph-union. We begin by discussing a well-known way for obtaining a coloring for the union of a finite number of graphs by making use of given colorings of the individual graphs. Let $k$ be a positive integer, let $G_{1}, \ldots, G_{k}$ be graphs on the same vertex set $V$. For each $i\in [k]$, let $f_{i}:V\to [k_{i}]$ be a $k_{i}$-coloring of $G_{i}$. Then the \defin{product coloring obtained from $f_{1}, \ldots, f_{k}$} is the function $\prod_{i\in[k]}f_{i}: V\to \{(j_{1}, \ldots, j_{k}): j_{i}\in [k_{i}], \text{ for every }i\in[k]\}$ which is defined as follows: $\prod_{i\in[k]}f_{i}(v):=(f_{1}(v), \ldots, f_{k}(v))$, for every $v\in V$. The following well-known result states that if the colorings of the individual graphs are proper, then the corresponding product coloring is proper as well:

\begin{proposition}[Folklore]
\label{prop:product.coloring.for.union}
    Let $k$ be a positive integer, let $G_{1}, \ldots, G_{k}$ be graphs on the same vertex set $V$, and for each $i\in [k]$, let $f_{i}:V\to [k_{i}]$ be a proper $k_{i}$-coloring of $G_{i}$. Then the product coloring obtained from $f_{1}, \ldots, f_{k}$ is a proper $\left(\prod_{i\in[k]}k_{i}\right)$-coloring of $G_{1}\cup \ldots \cup G_{k}$. In particular, $\chi(\cup_{i\in[k]}G_{i}) \leq \prod_{i\in[k]} \chi(G_{i})$.
\end{proposition}

The following, which states that graph-union preserves $\chi$-boundedness, is an immediate corollary of \autoref{prop:product.coloring.for.union} and of the fact that $\omega(G_{1}\cup \ldots \cup G_{k}) \geq \max_{i\in[k]}\{\omega(G_{i})\}$:

\begin{proposition}[{Gy\'arf\'as \cite[Proposition 5.1 (a)]{gyarfas1985problems}}]
    If $\mathcal{G}_{1}, \ldots \mathcal{G}_{k}$ are $\chib$ classes of graphs with $\chi$-bounding functions $f_{1}, \ldots, f_{k}$ respectively, then $\mathcal{G}_{1} \gcup \ldots \gcup \mathcal{G}_{k}$ is a $\chib$ family and $\prod_{i\in [k]} f_{i}$ is a suitable $\chi$-bounding function.
\end{proposition}

What about graph-intersection?

\begin{question}
\label{q.gintersetcion.of.chib}
    Is it always true that the graph-intersection of two $\chib$ classes of graphs is $\chib$?
\end{question}

\begin{question}
\label{q.self-gintersetcion.of.chib}
    Is it always true that if a class $\mathcal{A}$ is $\chib$, then for every positive integer $k$ the class $\gcap_{i\in [k]}\mathcal{A}$ is $\chib$?
\end{question}

 Early constructions of geometric intersection graphs show that both of the above questions have negative answers. Given a finite family of nonempty sets $\mathcal{A}$, the \defin{intersection graph} of $\mathcal{A}$ is the graph with vertex set $\mathcal{A}$, in which two vertices are adjacent if and only if they have a non-empty intersection. Classes of intersection graphs of geometric figures were among the first to be studied from the perspective of $\chi$-boundedness (see \cite{kostochka2004coloring} for a survey).
 
In 1960 Asplund and Gr\"unbaum \cite{asplund1960coloring} proved that the class of intersection graphs of axis-parallel rectangles in the plane is polynomially $\chi$-bounded. Surprisingly, the situation changes in $\mathbb{R}^{3}$. In 1965, Burling \cite{burling1965coloring} proved, by exhibiting a construction, that the class of intersection graphs of axis-parallel boxes in $\mathbb{R}^{3}$ contains triangle-free graphs of arbitrarily large chromatic number, and thus is not $\chi$-bounded. An \defin{interval graph} is a graph which is isomorphic to the intersection graph of a family of intervals on a linearly ordered set (such as the real line). As Bielecki \cite{bielecki1948} and Rado \cite{rado1948covering} proved in the late 1940's, interval graphs are perfect. We note that the class of intersection graphs of axis-parallel boxes in $\mathbb{R}^{k}$ is exactly the class $\gcap_{i\in[k]}\mathcal{I}$: Observe that given an intersection graph of axis-parallel boxes in $\mathbb{R}^{k}$, say $G$, the set of the projections of the boxes in one of the $k$ axes gives rise to an interval graph, and it is easy to see that $G$ is the intersection of these $k$ interval graphs. For the other inclusion the argument is similar. Thus, the class of intersection graphs of axis-parallel rectangles in the plane is the class $\mathcal{I}\gcap \mathcal{I}$, and the class of intersection graphs of axis-parallel boxes in $\mathbb{R}^{3}$ is the class $\mathcal{I}\gcap \mathcal{I}\gcap \mathcal{I}$. Hence, Asplund and Gr\"unbaum \cite{asplund1960coloring} proved that the class $\mathcal{I}\gcap \mathcal{I}$ is $\chib$, and Burling \cite{burling1965coloring} proved that the class $\mathcal{I}\gcap \mathcal{I}\gcap \mathcal{I}$ is not $\chib$. It follows that both \autoref{q.gintersetcion.of.chib} and \autoref{q.self-gintersetcion.of.chib} have a negative answer.

Let $\mathcal{A}$ be a class of graphs. We call $\mathcal{A}$ \defin{intersectionwise $\chi$-imposing} if for every class of graphs $\mathcal{B}$ the class $\mathcal{A}\gcap \mathcal{B}$ is $\chib$. We call $\mathcal{A}$ \defin{\g} if for every $\chib$ class of graphs $\mathcal{B}$ the class $\mathcal{A}\gcap \mathcal{B}$ is $\chib$. Finally, we call $\mathcal{A}$ \defin{\sg} if for every positive integer $k$ the class $\gcap_{i\in [k]}\mathcal{A}$ is $\chib$. 

\subsection{Organization of this paper}

The rest of this paper is organized as follows: In \autoref{sec:imposing} we prove a characterization of intersectionwise $\chi$-imposing graph classes. In \autoref{sec:decomposable.graphs.guarding} we define a certain type of a graph decomposition and we prove that if the graphs of a class admit such a decomposition, then this class is \g{}. We use this result to prove that the classes of unit interval graphs and of line graphs of bipartite graphs are \g{}. In \autoref{sec:constructions} we prove that several $\chib$ classes are not \g; these results give rise to necessary conditions on a finite set $\mathcal{H}$ of graphs for the class of $\mathcal{H}$-free graphs to be \g{}. We summarize these conditions in \autoref{subsec:finite.sets.conditions}. In \autoref{sec:H-free}, we characterize the \g{} classes which are defined by a single forbidden induced subgraph. We note that in an upcoming paper we discuss results towards a complete characterization of \g{} classes which are defined by two forbidden induced subgraphs. Finally in \autoref{sec:self.chi} we investigate the \sg{} classes which are defined by a single forbidden induced subgraph; we conjecture a characterization of these classes; we prove ways to construct \sg{} classes from existing \g{} classes; and we discuss some open problems for \sg{} classes.

\section{A characterization of intersectionwise \texorpdfstring{$\chi$}{TEXT}-imposing graph classes}
\label{sec:imposing}

Following \cite{chudnovsky2021induced}, we say that $\mathcal{A}$ is \defin{colorable} if there exists an integer $k$ such that every graph in $\mathcal{A}$ has chromatic number at most $k$. The main result of this section is the following:

\begin{theorem}
\label{thm:chi_imposing.characterization}
    Let $\mathcal{C}$ be a class of graphs. Then $\mathcal{C}$ is intersectionwise $\chi$-imposing if and only if $\mathcal{C}$ is colorable.
\end{theorem}

We begin with the following observation:

\begin{observation}
    \label{obs:chig.is.chib}
    Let $\mathcal{C}$ be a \g{} class of graphs. Then $\mathcal{C}$ is $\chib$.
\end{observation}

\begin{proof}[Proof of \autoref{obs:chig.is.chib}]
    Since $\mathcal{C}$ is \g{}, it follows that the graph-intersection of $\mathcal{C}$ with the class of complete graphs is $\chib$. Since this graph-intersection contains the class $\mathcal{C}$, it follows that $\mathcal{C}$ is $\chib$ as well.
\end{proof}

We are now ready to prove \autoref{thm:chi_imposing.characterization}:

\begin{proof}[Proof of \autoref{thm:chi_imposing.characterization}]
    For the forward direction: Since every intersectionwise $\chi$-imposing class is \g{}, it follows by \autoref{obs:chig.is.chib} that $\mathcal{C}$ is $\chib$. Let $f$ be a $\chi$-bounding function for $\mathcal{C}$. We claim that $\mathcal{C}$ does not contain arbitrarily large complete graphs. Suppose not. Then the graph-intersection of $\mathcal{C}$ with the class of all graphs contains all graphs. Thus, since $\mathcal{C}$ is intersectionwise $\chi$-imposing, the class of all graphs is $\chi$-bounded which is a contradiction. Let $\omega(\mathcal{C})$ be the maximum size of a complete graph in $\mathcal{C}$, and let $k:= f(\omega(\mathcal{C}))$. Then for every graph $G\in \mathcal{C}$, we have that $\chi(G) \leq k$. 

    For the backward direction: Let $k$ be a positive integer such that $\chi(G)\leq k$ for every graph $G\in \mathcal{C}$. Let $\mathcal{A}$ be a class of graphs, and let $H\in \mathcal{C}\gcap \mathcal{A}$. Then $H$ is a subgraph of a graph in $\mathcal{C}$, and thus $\chi(H)\leq k$. Hence, the class $\mathcal{C}\gcap \mathcal{A}$ is $\chib$ by the function $f(\omega)=k$. This completes the proof of \autoref{thm:chi_imposing.characterization}.
\end{proof}

\section{Unions of graphs of bounded componentwise \texorpdfstring{$r$}{TEXT}\hyph{dependent} chromatic number}
\label{sec:decomposable.graphs.guarding}

In this section, we prove that classes of graphs which admit a certain decomposition are \g{}, and we use this result to prove that the classes of unit interval graphs and of line graphs of bipartite graphs are \g{} (we define these classes later on this section). We first need some definitions.

Let $G$ be a graph and let $r\geq 2$ be an integer. We say that $G$ is \defin{componentwise $r$-dependent} if every component of $G$ has independence number at most $r-1$. Thus the componentwise $r$-dependent chromatic number of a graph $G$ is the minimum integer $k$ for which $G$ admits a componentwise $r$-dependent $k$-coloring. We say that $G$ is \defin{$(t,k,r)$-decomposable} if there exist positive integers $t, k$ and $r$ such that $G$ is the union of $t$ graphs of componentwise $r$-dependent chromatic number at most $k$. Finally we say that a class of graphs $\mathcal{C}$ is \defin{decomposable} if there exist $t,k$ and $r$ such that every graph in $\mathcal{C}$ is $(t,k,r)$-decomposable. The main result of this section is the following:

\begin{theorem}
    \label{thm:decomposable.graphs.chi.guarding}
    Let $\mathcal{C}$ be a decomposable class of graphs. Then $\mathcal{C}$ is \g{}.
\end{theorem}

The main ingredient for the proof of \autoref{thm:decomposable.graphs.chi.guarding} is following result on classes of graphs of bounded independence number:

\begin{lemma}
\label{lem:rK1-free.guarding}
    Let $r$ be a positive integer. Then the class of $rK_1$-free graphs is \g{}.
\end{lemma}

For our proof of \autoref{lem:rK1-free.guarding} we need the following observation:

\begin{proposition}
\label{prop:bound.clique.of.H}
    Let $\mathcal{C}$ be a class of graphs for which there exists a function $g$ such that for every class $\mathcal{H}$ we have that for every $G\in \mathcal{C}$ and for every $H\in \mathcal{H}$ the following holds: $\omega(H)\leq g(\omega(G\cap H))$. Then $\mathcal{C}$ is \g{}.
\end{proposition}

\begin{proof}[Proof of \autoref{prop:bound.clique.of.H}]
    Let $\mathcal{H}$ be a $\chib$ class of graphs, let $f$ be a $\chi$-bounding function for $\mathcal{H}$, let $G\in \mathcal{C}$, and let $H\in \mathcal{H}$. Then $\chi(G\cap H) \leq \chi(H) \leq f(\omega (H)) \leq f(g(\omega(G\cap H)))$. Thus, $f\circ g$ is a $\chi$-bounding function for the class $\mathcal{C}\gcap \mathcal{H}$. Hence, $\mathcal{C}$ is \g{}. This completes the proof of \autoref{prop:bound.clique.of.H}.
\end{proof}

We also need the following version of Ramsey's theorem:

\begin{theorem}[{Ramsey \cite[Theorem B]{ramsey1930problem}}]
    \label{thm:ramsey_cliqueORindependent}
    Let $s$ and $t$ be positive integers. Then there exists an integer $n(s,t)$ such that if $G$ is a graph on $n(s,t)$ vertices, then $G$ contains either a clique of size $s$ or an independent set of size $t$.
\end{theorem}

The \defin{Ramsey number} $R(s,t)$ is the minimum integer such that every graph on $R(s,t)$ vertices contains either a clique of size $s$ or an independent set of size $t$.

\begin{proof}[Proof of \autoref{lem:rK1-free.guarding}]
    Let $\mathcal{C}$ be the class of $rK_1$-free graphs and let $\mathcal{D}$ be a $\chi$-bounded class of graphs with $\chi$-bounding function $f:\mathbb{N} \to \mathbb{N}$. Consider two graphs $G \in \mathcal{C}$ and $H \in \mathcal{D}$. We may assume that $V(G \cap H) = V(G) = V(H).$

    \begin{claim}
    \label{cl:rK1.bound.H.clique.number}
        $\omega(H) < R(\omega(G \cap H)+1, r)$.
    \end{claim}

    \begin{subproof}[Proof of \autoref{cl:rK1.bound.H.clique.number}]
        Suppose not. That is, there exists a clique $S \subseteq V(H)$ such that $|S| \geq R(\omega(G\cap H) + 1,r)$. Then $(G \cap H)[S] \cong G[S]$. Since $G$ is $rK_1$-free, it follows that $\omega(G \cap H) \geq \omega((G \cap H)[S]) \geq \omega(G \cap H) + 1$, a contradiction.
    \end{subproof}
Now \autoref{lem:rK1-free.guarding} follows by the above claim and \autoref{prop:bound.clique.of.H}. This completes the proof of \autoref{lem:rK1-free.guarding}.
\end{proof}

Now in order to prove that decomposable classes of graphs are \g{} we prove that all the operations needed to create the graphs of a decomposable class starting from $rK_{1}$-free graphs preserve the property of being \g{}.

\begin{proposition}
\label{prop:componentwise.guarding.is.guarding}
    Let $\mathcal{C}_{1}, \ldots ,\mathcal{C}_{t}$ be \g{} classes, and let $\mathcal{C}$ be a class of graphs such that for every $G\in \mathcal{C}$ and for every component $C$ of $G$ there exists $i\in [t]$ such that $C\in \mathcal{C}_{i}$. Then $\mathcal{C}$ is \g{}.
\end{proposition}

\begin{proof}[Proof of \autoref{prop:componentwise.guarding.is.guarding}]
    Let $\mathcal{H}$ be a $\chib$ class of graphs. For each $j\in [t]$, let $f_{j}$ be a $\chi$-bounding function of the class $\mathcal{C}_{j}\gcap \mathcal{H}$. Let $H\in \mathcal{H}$, let $G\in \mathcal{C}$, and let $C_{1}, \ldots, C_{l}$ be the components of $G$. Let $i\in [l]$ be such that $\chi(G_{i}\cap H) \geq \chi(G_{k}\cap H)$ for every $k\in [l]$. Then we have that $\chi(G\cap H) = \chi(G_{i}\cap H) \leq f_{i}(\omega(G_{i}\cap H))\leq \sum_{j\in[t]} f_{j}(\omega(G\cap H))$. Thus, $\sum_{j\in[t]} f_{j}$ is a $\chi$-bounding function of the class $\mathcal{C}\gcap\mathcal{H}$ and hence $\mathcal{C}$ is \g{}. This completes the proof of \autoref{prop:componentwise.guarding.is.guarding}.
\end{proof}

The following is an immediate corollary of \autoref{lem:rK1-free.guarding} and \autoref{prop:componentwise.guarding.is.guarding}:

\begin{corollary}
\label{cor:componentwise.r.dependent.guarding}
    Let $r\geq 2$ be an integer and $\mathcal{C}$ be a class of componentwise $r$-dependent graphs. Then $\mathcal{C}$ is \g{}.
\end{corollary}

\begin{proposition}
\label{prop:bounded.guarding.chromatic.number.is.guarding}
    Let $k$ be a positive integer, let $\mathcal{C}_{1}, \ldots ,\mathcal{C}_{t}$ be \g{} classes, and let $\mathcal{C}$ be a class of graphs such every $G\in \mathcal{C}$ has a $k$-coloring $f$ such that for every $j\in[k]$ there exists $i\in [t]$ such that $G[f^{-1}(j)] \in{C}_{i}$. Then $\mathcal{C}$ is \g{}.
\end{proposition}

\begin{proof}[Proof of \autoref{prop:bounded.guarding.chromatic.number.is.guarding}]
    Let $\mathcal{H}$ be a $\chib$ class of graphs. For each $i\in [t]$, let $f_{i}$ be a $\chi$-bounding function of the class $\mathcal{C}_{i}\gcap \mathcal{H}$. Let $H\in \mathcal{H}$, let $G\in \mathcal{C}$, and let $f$ be a $k$-coloring $f$ such that for every $j\in[k]$ there exists $i\in [t]$ such that $G[f^{-1}(j)] \in{C}_{i}$. Let $j\in[k]$ be such that $\chi(G[f^{-1}(j)]\cap H)\geq \chi(G[f^{-1}(l)]\cap H)$ for every $l\in[k]$. Let $i\in[t]$ be such that $G[f^{-1}(j)]\in \mathcal{C}_{i}$. Then we have that $\chi(G\cap H) \leq k\cdot\chi(G[f^{-1}(j)]\cap H) \leq k\cdot f_{i}(\omega\left(G[f^{-1}(j)]\cap H\right)) \leq k\cdot f_{i}(\omega(G\cap H)) \leq k\cdot \sum_{p\in[t]}f_{p}(\omega(G\cap H))$. Thus, $k\cdot \sum_{p\in[t]}f_{p}$ is a $\chi$-bounding function of the class $\mathcal{C}\gcap\mathcal{H}$ and hence $\mathcal{C}$ is \g{}. This completes the proof of \autoref{prop:bounded.guarding.chromatic.number.is.guarding}.
\end{proof}

The following is an immediate corollary of \autoref{cor:componentwise.r.dependent.guarding} and \autoref{prop:bounded.guarding.chromatic.number.is.guarding}:

\begin{corollary}
\label{cor:bouned.componentwise.chi.is.guarding}
    Let $k$ and $r\geq 2$ be positive integers and let $\mathcal{C}$ be a class of graphs of componentwise $r$-dependent chromatic number at most $k$. Then $\mathcal{C}$ is \g{}.
\end{corollary}

By the above, in order to prove \autoref{thm:decomposable.graphs.chi.guarding}, which states that decomposable classes are \g{}, it suffices to prove that graph-union preserves the property of being \g{}.

\begin{proposition}
\label{prop:union.of.guarding.is.guarding}
    Let $\mathcal{C}_{1}, \ldots , \mathcal{C}_{t}$ be \g{} classes. Then their graph-union is \g{}.
\end{proposition}

\begin{proof}[Proof of \autoref{prop:union.of.guarding.is.guarding}]
    Let $\mathcal{C}$ be the graph-union of $\mathcal{C}_{1}, \ldots , \mathcal{C}_{t}$. Let $G\in \mathcal{C}$. Let $\mathcal{H}$ be a $\chib$ class and let $H\in \mathcal{H}$. For each $i\in [t]$, let $G_{i}\in \mathcal{C}_{i}$ be such that $G=\cup_{i\in [t]} G_{i}$, and let $f_{i}$ be a $\chi$-bounding function for the class $\mathcal{C}_{i}\gcap \mathcal{H}$. Then we have that $\chi(G\cap H) = \chi((\cup_{i\in [t]} G_{i})\cap H) = \chi(\cup_{i\in [t]} (G_{i}\cap H))$. Since, by \autoref{prop:product.coloring.for.union}, we have that $ \chi(\cup_{i\in [t]} (G_{i}\cap H)) \leq \prod_{i\in[t]}\chi(G_{i}\cap H)$, it follows that $\chi(G\cap H)\leq \prod_{i\in[t]}\chi(G_{i}\cap H) \leq \prod_{i\in[t]} f_{i}(\omega(G_{i}\cap H)) \leq \prod_{i\in[t]} f_{i}(\omega(G\cap H))$ Thus, $\prod_{i\in[t]} f_{i}$ is a $\chi$-bounding function of the class $\mathcal{C}\gcap\mathcal{H}$ and hence $\mathcal{C}$ is \g{}. This completes the proof of \autoref{prop:union.of.guarding.is.guarding}.
\end{proof}

Now \autoref{thm:decomposable.graphs.chi.guarding} is an immediate corollary of \autoref{cor:bouned.componentwise.chi.is.guarding} and \autoref{prop:union.of.guarding.is.guarding}.

A \defin{unit interval} graph is an interval graph which has a representation in which all the intervals have length one.
 
\begin{lemma}
\label{lem:unit.interval.decomposable}
    The class of unit interval graphs is decomposable.
\end{lemma}

\begin{proof}[Proof of \autoref{lem:unit.interval.decomposable}]
    In what follows we prove that every unit interval graph has componentwise $2$-dependent chromatic number at most two.
    
    Let $\{I_{1}, \ldots, I_{n}\}$ be a family of intervals on the real line such that each has unit length. We may assume that no interval has an integer as an endpoint, and thus every interval contains exactly one integer. Let $\{A,B\}$ be the partition of $\{I_{1}, \ldots, I_{n}\}$ which is defined as follows: $A$ contains an interval $I$ if and only if $I$ contains an even integer, and $B=\{I_{1}, \ldots, I_{n}\}\setminus A$. Thus $B$ contains exactly those intervals which contain an odd integer. Let $G$ be the intersection graph of $\{I_{1}, \ldots, I_{n}\}$.

    Finally we claim that each of $G[A]$ and $G[B]$ is a componentwise $2$-dependent graph, that is, a disjoint union of complete graphs. The claim for follows immediately by the observation that any two vertices in $G[A]$ (respectively in $G[B]$) are adjacent if and only if the two corresponding intervals contain the same even (respectively odd) integer. This completes the proof of \autoref{lem:unit.interval.decomposable}.
\end{proof}

The following is an immediate corollary of \autoref{thm:decomposable.graphs.chi.guarding} and \autoref{lem:unit.interval.decomposable}.

\begin{corollary}
\label{cor:unit.interval.guarding}
    The class of unit interval graphs is \g{}.
\end{corollary}

Let $G$ be a graph. The \defin{line graph} of $G$, which we denote by $L(G)$, is the graph with vertex set the set $E(G)$ and edge set the set $\{ef:e\cap f\neq \emptyset\}$.

\begin{lemma}
\label{lem:line.graphs.of.bipartite.decomposable}
   The class of line graphs of bipartite graphs is decomposable.
\end{lemma}

\begin{proof}[Proof of \autoref{lem:line.graphs.of.bipartite.decomposable}]
    In what follows we prove that each line graph of a bipartite graph is the union of two componentwise $2$-dependent graphs.
    
    Let $G$ be a bipartite graph, and let $\{A_{1}, A_{2}\}$ be a bipartition of $V(G)$. Let $\{E_{1}, E_{2}\}$ be a partition of $E(L(G))$ which is defined as follows: for each $i\in [2]$, an edge $ef$ of $L(G)$ is in $E_{i}$ if and only if $e\cap f\subseteq A_{i}$. 

    We claim that for each $i\in [2]$ the graph $(E(G), E_{i})$ is componentwise $2$-dependent, that is, it is the disjoint union of complete graphs. Indeed, let $e,f, g \in E(G)$ and suppose that $ef, fg\in E_{i}$. Let $v$ be the unique element of $f\cap A_{i}$. Then $v\in e\cap g$, and thus $eg\in E_{i}$. Hence the adjacency relation in $(E(G), E_{i})$ is an equivalence relation. This completes the proof of \autoref{lem:line.graphs.of.bipartite.decomposable}. This completes the proof of \autoref{lem:line.graphs.of.bipartite.decomposable}.
\end{proof}

The following is an immediate corollary of \autoref{thm:decomposable.graphs.chi.guarding} and \autoref{lem:line.graphs.of.bipartite.decomposable}:

\begin{corollary}
    The class of line graphs of bipartite graphs is \g{}.
\end{corollary}

In \autoref{subsec:line.graphs.of.graphs.of.large.girth} we will show that the class of line graphs is not \g{}.

\section{Classes which are not \texorpdfstring{\g}{TEXT}}
\label{sec:constructions}

In this section, we prove that certain $\chib$ classes of graphs are not \g{}. We do this by considering the graph-intersections of pairs of $\chib$ graph classes and proving that these graph-intersections contain triangle-free graphs of arbitrarily large chromatic number. We then use these results to prove certain properties that a finite set of graphs $\mathcal{H}$ should satisfy for the class of $\mathcal{H}$-free graphs to be \g{}.

\subsection{Line graphs of graphs of large girth and complete multipartite graphs}
\label{subsec:line.graphs.of.graphs.of.large.girth}

 The \defin{chromatic index} of a graph $G$, denoted by $\chi'(G)$, is the minimum size of a partition of the edge set of $G$ into matchings. Vizing \cite{vizing1964estimate} proved that for every graph $G$, we have $\chi'(G) \leq \Delta(G) +1$. Hence, we have the following:

\begin{proposition}[{Vizing \cite{vizing1964estimate}}]
\label{prop:line.graphs.chib.vizing}
    The class of line graphs is $\chib$ by the function $f(\omega) = \omega +1$.
\end{proposition}

Let $A$ and $B$ be disjoint subsets of $V(G)$. We say that $A$ and $B$ are \defin{complete to each other}  (respectively \defin{anticomplete to each other}) if for every $a\in A$ and $b\in B$ we have $ab\in E(G)$ (respectively $ab\notin E(G)$). The graph $G$ is \defin{complete $k$-partite} if its vertex set can be partitioned into a family of $k$ non-empty independent sets which are pairwise complete to each other. Finally, we say that $G$ is \defin{complete multipartite} if there exists a positive integer $k$ such that $G$ is a complete $k$-partite graph. It follows from the definition of complete multipartite graphs that a graph $G$ is complete multipartite if and only if the non-adjacency relation on $V(G)$ is an equivalence relation. It is straightforward to show that assigning each equivalence class its own color is an optimal coloring, and hence that complete multipartite graphs are perfect. 

The main results of this subsection are the following two theorems:

\begin{theorem}
\label{thm:complete.multipartite.graphs.not.chig}
    The class of complete multipartite graphs is not \g{}.
\end{theorem}

\begin{theorem}
\label{thm:line.graphs.not.chig}
    Let $g \geq 3$. Then the class of line graphs of graphs of girth at least $g$ is not \g{}. In particular, the class of line graphs is not \g{}.
\end{theorem}

We prove \autoref{thm:complete.multipartite.graphs.not.chig} and~\autoref{thm:line.graphs.not.chig} by showing that the graph intersection of the class of line graphs of graphs of girth at least $g$ for $g \geq 3$ and the class of complete multipartite graphs contain triangle-free graphs of arbitrarily large chromatic number. We first introduce notions and results we will need to describe the construction of these graphs.

Let $D$ be a digraph. We denote by $\chi(D)$ (respectively by $L(D)$) the chromatic number (respectively the line graph) of the underlying undirected graph of $D$. Following Harary and Norman \cite{harary1960line_digraphs}, the \defin{line digraph} of $D$, which we denote by $\vec{L}(D)$, is the digraph with $V(\vec{L}(D))=E(D)$ and $E(\vec{L}(D))=\{(uv)(vw):uv,vw\in E(D)\}$. We remark that the underlying undirected graph of $\vec{L}(D)$ is a subgraph of $L(D)$.

We need the following theorem which states that the chromatic number of the underlying undirected graph of the line digraph of a digraph $D$ is lower-bounded by a function of $\chi(D)$:

\begin{theorem}[{Harner and Entringer~\cite[Theorem 9]{coloring_directed_line_graph}}]
\label{thm:coloring.directed.line.graphs}
    Let $D$ be a digraph. Then $\chi(\vec{L}(D)) \geq \log_2(\chi(D))$.
\end{theorem}

Let $n$ and $k$ be integers such that $n>2k>2$. Erd\H{o}s and Hajnal~\cite{erdos1968chromatic} defined the \defin{shift graph} $G(n,k)$ as the graph with vertex set the set of all $k$-tuples $(t_{1},\ldots,t_{k})$ such that $1\leq t_{1} < \ldots < t_{k} \leq n$, in which the vertices $(t_{1},\ldots,t_{k})$ and $(t_{1}',\ldots,t_{k}')$ are adjacent if and only if $t_{i+1}=t'_{i}$ for $1 \leq i<k$, or vice versa. Observe that $G(n,k)$ is triangle-free. We additionally consider \defin{directed shift graphs}, orientations of shift graphs where $(t_1, \ldots, t_k)(t_1', \ldots, t_k')$ is an arc exactly if $t_{i+1}=t'_{i}$ for all $1 \leq i<k$. For integers $n$ and $k$ such that $n > 2k > 2$, we denote the directed shift graph of $k$-tuples over an alphabet of size $n$ by $\vec{G}(n,k)$. The shift graph $G(n,k)$ is thus the underlying undirected graph of $\vec{G}(n,k)$.

Erd\H{o}s and Hajnal~\cite{erdos1968chromatic} proved the following result for the chromatic number of shift graphs:

\begin{theorem}[Erd\H{o}s and Hajnal~\cite{erdos1968chromatic}]
\label{thm:erdos.hajnal.shift}
    Let $n$ and $k$ be integers such that $n>2k>2$. Then $\chi(G(n,k))=(1-o(1))\log^{(k-1)}(n)$. In particular, $\chi(G(n,2))=\lceil \log n\rceil$, and for $n$ sufficiently large we have that $\chi(G(n,3))>k$.
\end{theorem}

Hence, for any fixed $k \geq 2$, the shift graphs $\left(G(n,k)\right)_{n \in \mathbb{N}^+}$ form a class of triangle-free graphs of arbitrarily large chromatic number. Any class containing such a class of shift graphs is thus not $\chi$-bounded.

In 2018, Gábor Tardos gave a talk at \textit{Combinatorics: Extremal, Probabilistic and Additive} in São Paulo about work of his and Bartosz Walczak on a conjecture of Erd\H{o}s and Hajnal. One of their main results is the following:

\begin{theorem}[Tardos \& Walczak]\label{thm:subgraph.shift.graph.large.girth.and.chi}
    Let $k \geq 2$, $g \geq 3$ and $c > 0$ be integers. Then there exists an integer $n^* > 0$ such that for each $n \geq n^*$, the shift graph $G(n, k)$ contains a subgraph $H_n$ with girth at least $g$ and $\chi(H_n) \geq c$.
\end{theorem}

We will use this result in our construction. A write-up of the proof of \autoref{thm:subgraph.shift.graph.large.girth.and.chi} may be found in Rodrigo Aparecido Enju's master's dissertation~\cite{Aparecido_Enju_2023} (in Portuguese).

We are now ready to describe the construction showing that the graph intersection of line graphs of graphs of large girth and complete multipartite graphs contains triangle-free graphs of arbitrarily large chromatic number.

\begin{lemma}\label{lem:intersection.line.graphs.complete.multipartite.not.chib}
    Let $g \geq 3$ be an integer, let $\mathcal{L}_g$ be the class of line graphs of graphs with girth at least $g$, and let $\mathcal{C}$ be the class of complete multipartite graphs. Then the class $\mathcal{L}_g \gcap \mathcal{C}$ is not $\chi$-bounded.
\end{lemma}
\begin{proof}[Proof of \autoref{lem:intersection.line.graphs.complete.multipartite.not.chib}]
    To show that $\mathcal{L}_g \gcap \mathcal{C}$ is not $\chi$-bounded, we will show that $\mathcal{L}_g \gcap \mathcal{C}$ contains triangle-free graphs of arbitrarily large chromatic number. As such, no $\chi$-bounding function $f: \mathbb{N} \to \mathbb{N}$ exists, since $f(2)$ cannot be defined satisfactorily. Let $c \geq 1$ be an integer. We will construct a triangle-free graph in $\mathcal{L}_g \gcap \mathcal{C}$ with chromatic number at least $c$.

    By \autoref{thm:subgraph.shift.graph.large.girth.and.chi}, there exists an integer $n > 0$ such that $G(n, 2)$ has a subgraph $H$ with girth at least $g$ and $\chi(H) \geq 2^c$. We consider the corresponding subdigraph $\vec{H}$ of $\vec{G}(n,2)$ with $V(\vec{H}) = V(H)$.

    Let $C$ be a complete multipartite graph defined as follows; set $V(C) := [n]^3$, and two vertices $(a_1, a_2, a_3)$ and $(b_1, b_2, b_3)$ are adjacent in $C$ if $a_2 \neq b_2$. The equivalence classes of $C$ are then exactly the vertices whose middle entries are equal.

    Let $G_{\vec{L}}$ be the underlying undirected graph of $\vec{L}(\vec{H})$. In the remainder of the proof, it is useful to consider a different representation of the vertices of line (di)graphs of (subgraphs of) (directed) shift graphs. For an edge (or arc) $(x,y)(y,z)$ in the (directed) shift graph, we will label the corresponding vertex in the line (di)graph with the tuple $(x, y, z)$. 

    \begin{claim}\label{claim:line.digraph.intersection.representation}
        $G_{\vec{L}} = L(\vec{H}) \cap C$.
    \end{claim}
    \begin{subproof}[Proof of \autoref{claim:line.digraph.intersection.representation}]
        First observe that $V(G_{\vec{L}}) = V(L(\vec{H}))$. Moreover, since each vertex in these two graphs is represented by a 3-tuple with elements in $[n]$, it follows that $V(G_{\vec{L}}), V(L(\vec{H})) \subseteq V(C)$. Hence, $V(G_{\vec{L}}) = V(L(\vec{H}) \cap C)$, as desired.

        By the definitions of shift graphs and line graphs, two vertices $(x, y, z)$ and $(x', y', z')$ in $L(\vec{H})$ are adjacent if
        \begin{enumerate}[topsep=0.2cm,itemsep=-0.4cm,partopsep=0.5cm,parsep=0.5cm]
            \item[(1)] $x = x'$ and $y = y'$ (both vertices represent arcs with tail $(x, y)$ in $\vec{H}$);
            \item[(2)] $y = y'$ and $z = z'$ (both vertices represent arcs with head $(y, z)$ in $\vec{H}$);
            \item[(3)] $x = y'$ and $y = z'$ (the two vertices represent arcs with $(x,y)$ as tail and head respectively in $\vec{H}$); or
            \item[(4)] $y = x'$ and $z = y'$ (the two vertices represent arcs with $(y,z)$ as head and tail respectively in $\vec{H}$).
        \end{enumerate}
        We note that in both cases (1) and (2), $y = y'$, and hence in $C$ the vertices $(x, y, z)$ and $(x', y', z')$ are not adjacent. In case (3), $y = z' \neq y'$, and in case (4), $y = x' \neq y'$, and hence the edges of both these cases are also present in $C$. Thus, $L(\vec{H}) \cap C$ is exactly the graph with vertex set $V(L(\vec{H}))$ and all edges satisfying (3) or (4). From the definition of line digraphs, it follows that this is exactly $G_{\vec{L}}$, as desired.
    \end{subproof}

    \begin{claim}\label{claim:line.digraph.triangle.free}
        $G_{\vec{L}}$ is triangle-free.
    \end{claim}
    \begin{subproof}[Proof of \autoref{claim:line.digraph.triangle.free}]
        Observe that $\vec{L}(\vec{G}(n,2)) \cong \vec{G}(n,3)$. Because $\vec{H} \subseteq \vec{G}(n, 2)$, we thus have that $L(\vec{G}(n,2)) \subseteq \vec{G}(n, 3)$, and hence $G_{\vec{L}} \subseteq G(n, 3)$. The claim then follows from the fact that shift graphs are triangle-free.
    \end{subproof}
    
    Since $H$ has girth at least $g$, and as $\vec{H}$ is an orientation of $H$, we have that $L(\vec{H}) \in \mathcal{L}_g$. Thus, by \autoref{claim:line.digraph.intersection.representation}, $G_{\vec{L}} \in \mathcal{L}_g \gcap \mathcal{C}$. Moreover, by \autoref{claim:line.digraph.triangle.free}, $G_{\vec{L}}$ is triangle-free. Finally, as $\chi(H) \geq 2^c$, by \autoref{thm:coloring.directed.line.graphs}, $\chi(G_{\vec{L}}) = \chi(\vec{L}(\vec{H})) \geq \log_2(\chi(L(\vec{H}))) = \log_2(\chi(L(H))) \geq c$, as desired. This completes the proof of \autoref{lem:intersection.line.graphs.complete.multipartite.not.chib}.
\end{proof}

\autoref{thm:complete.multipartite.graphs.not.chig} and \autoref{thm:line.graphs.not.chig} now follow from \autoref{lem:intersection.line.graphs.complete.multipartite.not.chib} and the fact that the classes of complete multipartite graphs and line graphs are both $\chi$-bounded.

The following is an immediate corollary of \autoref{thm:complete.multipartite.graphs.not.chig}:

\begin{corollary}
\label{cor:H.finite.guarding.complete.multipartite}
    Let $\mathcal{H}$ be a finite set of graphs. If the class of $\mathcal{H}$-free graphs is \g{}, then $\mathcal{H}$ contains a complete multipartite graph.
\end{corollary}

To obtain a similar corollary of \autoref{thm:line.graphs.not.chig}, we observe the following:

\begin{observation}
\label{obs:avoid.line.graphs.of.graphs.of.large.g}
    Let $\mathcal{H}$ be finite set of graphs which contains no line graph of a forest. Then there exists $g$ such that the class of all $\mathcal{H}$-free graphs contains all line graphs of graphs of girth at least $g$.
\end{observation}

The following is then an immediate corollary of \autoref{thm:line.graphs.not.chig} and \autoref{obs:avoid.line.graphs.of.graphs.of.large.g}:

\begin{corollary}
\label{cor:H.finite.guarding.line.of.forest}
    Let $\mathcal{H}$ be a finite set of graphs. If the class of $\mathcal{H}$-free graphs is \g{}, then $\mathcal{H}$ contains the line graph of a forest.
\end{corollary}

We use $K_{s,t}$ to denote a complete bipartite graph with parts of sizes $s$ and $t$. A \defin{star} is a $K_{1,t}$ for some positive integer $t$, and a \defin{claw} is a $K_{1,3}$.

\begin{corollary}
\label{cor:line.graph.guarding}
    The class of claw-free graphs is not \g{}. 
\end{corollary}

\begin{proof}[Proof of \autoref{cor:line.graph.guarding}]
    As shown by Van~Rooij and Wilf~\cite{van1965interchange}, line graphs are claw-free, thus the class of claw-free graphs contains all line graphs. Hence, by \autoref{thm:line.graphs.not.chig}, the class of claw-free graphs is not \g{}. This completes the proof of \autoref{cor:line.graph.guarding}.
\end{proof}

\subsection{Trivially perfect graphs}
\label{subsec:burling.strongly.chordal.trivially.perfect}

Following Golumbic \cite{golumbic1978trivially} we say that a graph $G$ is \defin{trivially perfect graph} if for every induced subgraph $H$ of $G$ the independence number of $H$ equals the number of maximal cliques of $H$. We note that the class of trivially perfect graphs is a subclass of the class of perfect graphs \cite{golumbic1978trivially} and thus $\chib$. We also note the class of trivially perfect graphs has the following nice characterization in terms of forbidden induced subgraphs: 

\begin{theorem}[{Golumbic \cite[Theorem 2]{golumbic1978trivially}}]
    Let $G$ be a graph. Then $G$ is trivially perfect if and only if $G$ is $\{P_{4}, C_{4}\}$-free.
\end{theorem}

The main result of this subsection is the following theorem:

\begin{theorem}
\label{thm:trivially.perfect.not.chig}
    The class of trivially perfect graphs is not \g{}.
\end{theorem}

We prove \autoref{thm:trivially.perfect.not.chig} by showing that the graph-intersection of the class of trivially perfect graphs with a $\chib$ class contains graphs of bounded clique number and arbitrarily large chromatic number. We need some definitions. A \defin{hole} in a graph is an induced subgraph which is a cycle of length at least four. A graph is \defin{chordal} if it contains no holes. Let $G$ be a graph, let $l\geq 3$, and let $C=\{v_{1}, \ldots, v_{2l}\}$ be an even cycle of length at least six in $G$. For $i,j\in[2l]$, we say that the edge $v_{i}v_{j}$ is an \defin{odd chord} of $C$ if $i$ and $j$ have different parity; in this case the distance of $v_{i}$ and $v_{j}$ in the cycle is odd. A \defin{strongly chordal graph} is a chordal graph in which every even cycle of length at least six has an odd chord. We note that, since chordal graphs are perfect \cite{berge1960problemes}, strongly chordal graphs are perfect. In what follows we show that the graph-intersection of the class of trivially perfect graphs with a proper subclass of the class of strongly chordal graphs is not $\chib$ by proving that this graph-intersection contains all Burling graphs (defined below).

In 1965, Burling \cite{burling1965coloring} introduced a sequence $\{\mathcal{B}_{k}\}_{k\geq 1}$ of families of axis-aligned boxes in $\mathbb{R}^{3}$. Throughout this section, for each $k$, we denote by $G_{k}$ the intersection graph of $\mathcal{B}_{k}$; Burling proved the following: 

\begin{theorem}[Burling \cite{burling1965coloring}]
\label{thm:Burling}
For every positive integer $k$, the graph $G_{k}$ is triangle-free and has chromatic number 
at least $k$.
\end{theorem}

 We call the sequence $\{G_{k}\}_{k\in \mathbb{N}}$ the \defin{Burling sequence}. Following Pournajafi and Trotignon \cite{pournajafi2023burling}, we say that a graph $G$ is a \defin{Burling graph} if there exists a positive integer $k$ such that $G$ is isomorphic to an induced subgraph of $G_{k}$. In what follows in this section we denote by $\mathcal{B}$ the hereditary class of Burling graphs. As we discussed in \autoref{sec:intro}, Asplund and Gr\"unbaum \cite{asplund1960coloring} proved that the $2$-fold graph-intersection of the class of interval graphs is $\chib$, and it follows from \autoref{thm:Burling} that the $3$-fold graph-intersection of the class of interval graphs is not $\chib$. Motivated by these results, Gy\'arf\'as \cite[Problem 5.7]{gyarfas1985problems} asked whether the graph-intersection of the class of chordal graphs with the class of interval graphs is $\chib$. Chaniotis, Miraftab, and Spirkl pointed out in \cite{chordality2024} that a result of Felsner, Joret, Micek, Trotter, and Wiechert \cite{felsner2017burling} implies that Burling graphs are contained in the graph-intersection of the class of chordal graphs with the class of interval graphs, and thus the answer to Gy\'arf\'as' question is negative. Here we strengthen this result, by proving the following:

\begin{theorem}
\label{thm:Burling.main}
The class of Burling graphs is a subclass of the graph-intersection of the class of trivially perfect graphs with a proper subclass of the class of strongly chordal graphs.
\end{theorem}

We note that \autoref{thm:trivially.perfect.not.chig} follows immediately from \autoref{thm:Burling.main}. Our proof of \autoref{thm:Burling.main} is based on a characterization of Burling graphs which was proved by Pournajafi and Trotignon in \cite{pournajafi2023burling}, where they defined the hereditary class of derived graphs and proved that this class is the same as the class of Burling graphs.

Before we proceed, we need to introduce some terminology and notation for rooted trees. A \defin{rooted tree} is a pair $(T,r)$ where $r$ is a vertex of $T$. We call $r$ the \defin{root} of $T$. If there is no danger of ambiguity, we use the notation $T$ for the rooted tree $(T,r)$. Let $(T,r)$ be a rooted tree. The \defin{parent} of a vertex $v\in V(T)\setminus \{r\}$, denoted by $p(v)$, is the neighbor of $v$ which lies in the unique $(v,r)$-path in $T$. If $p(v)=u$, then we say that $v$ is a \defin{child} of $u$. Let $u,v\in V(T)$. We say that $u$ and $v$ are \defin{siblings} if $p(u)=p(v)$. We say that $u$ is an \defin{ancestor} of $v$ if $u$ lies in the in the unique $(v,r)$-path in $T$. The \defin{descendants} of a vertex $u$ are all the vertices which have $u$ as an ancestor. Finally, following \cite{pournajafi2023burling} we say that a \defin{branch} in $T$ is a path $v_{1},\ldots, v_{k}$ such that for each $i\in [k-1]$ the vertex $v_{i}$ is the parent of the vertex $v_{i+1}$; in this case we say that the branch \defin{starts} at $v_{1}$ and ends at $v_{k}$. A \defin{principal branch} is a branch which starts at the root and ends at a leaf of the tree $T$. Following \cite{pournajafi2023burling} we say that a \defin{Burling tree} is a $4$-tuple $(T,r,l,c)$ in which:
\begin{enumerate}[label=(\roman*)]
    \item $T$ is a rooted tree and $r$ is its root;
    \item $l$ is a function associating to each non-leaf vertex $v$ of $T$ 
    one child of $v$ which is called the last-born of $v$;
    \item $c$ is a function defined on the vertices of $T$. If $v$ is a non-last-born vertex of $T$ other than the root, then $c$ associates to $v$ the vertex set of a (possibly empty) branch in $T$ starting at the last-born of $p(v)$. If $v$ is a last-born vertex or the root of $T$, then we define $c(v)=\emptyset$. We call $c$ the \defin{choose 
    function} of $T$.
\end{enumerate}

Following \cite{pournajafi2023burling} we say that the \defin{oriented graph fully derived} from the Burling tree $(T,r,l,c)$, which we denote by $A(T)$, is the oriented graph whose vertex set is $V(T)$ and $uv \in E(A(T))$ if and only if $v$ is a vertex in $c(u)$. The \defin{graph fully derived} from the Burling tree $(T,r,l,c)$, which we denote by $G(T)$, is the underlying undirected graph of $A(T)$. A graph (respectively oriented graph) is \defin{derived} from the Burling tree $(T,r,l,c)$ if it is an induced subgraph of $G(T)$ (respectively $A(T)$). A graph $G$ (respectively oriented graph $A$) is called a \defin{derived graph} (respectively \defin{oriented derived graph}) if there exists a Burling tree $(T,r,l,c)$ such that $G$ (respectively $A$) is derived from $(T,r,l,c)$.

\begin{theorem}[Pournajafi and Trotignon {\cite[Theorem 4.9]{pournajafi2023burling}}]
    \label{thm:Burling.same.as.derived}
    The class of derived graphs is the same as the class of Burling graphs.
\end{theorem}

By the above, in order to prove \autoref{thm:Burling.main}, it suffices to prove the following:

\begin{theorem}
    \label{thm:Burling.derived.in.intersection.}
    The class of derived graphs is a subclass of the graph-intersection of the class of trivially perfect graphs with a proper subclass of the class of strongly chordal graphs.
\end{theorem}

Let $(T,r,l,c)$ be a Burling tree. We denote by $C(T)$ the graph which we obtain from $G(T)$ by adding the necessary edges in order to make the vertex set of every principal branch of $T$ a clique. We also denote by $I(T)$ the graph that we obtain from $G(T)$ by adding edges so that every vertex is adjacent to all of its siblings, and to all the descendants of its last-born sibling.

\begin{observation}
\label{obs:derived.in.the.intersection}
    Let $(T,r,l,c)$ be a Burling tree. Then $G(T) = C(T)\cap I(T)$. 
\end{observation}

In what follows in this section, we denote by $\mathcal{C}$ (respectively by $\mathcal{I}$) the closure under induced subgraphs of the class $\{C(T):(T,r,l,c) \text{ is a Burling tree}\}$ (respectively the closure under induced subgraphs of the class $\{I(T):(T,r,l,c) \text{ is a Burling tree}\}$). We call the class $\mathcal{C}$ the class of \defin{Burling strongly chordal graphs}.

\begin{observation}
\label{obs:derived.is.equal.to.intersection}
    The class of derived graphs is contained in the class $\mathcal{C}\gcap \mathcal{I}$.
\end{observation}

By \autoref{obs:derived.is.equal.to.intersection}, in order to prove \autoref{thm:Burling.derived.in.intersection.}, it suffices to prove that $\mathcal{C}$ is a proper subclass of the class of strongly chordal graphs, and that $\mathcal{I}$ is the class of trivially perfect graphs.

To this end, we need to introduce some terminology in order to state a characterization of strongly chordal graphs. Let $G$ be a graph and let $u,v\in V(G)$. Following Farber \cite{farber1983characterizations}, we say $u$ and $v$ are \defin{compatible} if $N[u]\subseteq N[v]$ or $N[v]\subseteq N[u]$, and that a vertex $v\in V(G)$ is \defin{simple} if the vertices in $N[v]$ are pairwise compatible. Farber \cite{farber1983characterizations} gave the following characterization of strongly chordal graphs.

\begin{theorem}[Farber {\cite[Theorem 3.3]{farber1983characterizations}}]
\label{thm:strongly.chordal.characterization.simple}
    A graph $G$ is strongly chordal if and only if every induced subgraph of $G$ has a simple vertex.
\end{theorem}

We begin with two lemmas that we need in order to prove that $\mathcal{C}$ is a subclass of the class of strongly chordal graphs.

\begin{lemma}
\label{lem:Burling.collinear.vertices.are.compatible.in.C(T)}
   Let $(T,r,l,c)$ be a Burling tree and let $u$ be an ancestor of $v$ in $T$. Then $u$ and $v$ are compatible in $C(T)$, in particular $N_{C(T)}[v] \subseteq N_{C(T)}[u]$. 
\end{lemma}

\begin{proof}[Proof of \autoref{lem:Burling.collinear.vertices.are.compatible.in.C(T)}]
    The set $N_{C(T)}[v]$ can be partitioned in the following three sets:
    \begin{itemize}[topsep=0.2cm,itemsep=-0.4cm,partopsep=0.5cm,parsep=0.5cm]
        \item The set $N_{1}$ which contains $v$, the ancestors of $v$, and the descendants of $v$;
        \item the set $N_{2} = \{w\in V(T): wv\in E(A(T))\}$; and
        \item the set $N_{3} = \{w\in V(T): vw\in E(A(T))\}$.
    \end{itemize}

    Since the vertex set of every branch of $T$ is a clique in $C(T)$, it follows that $N_{1}\subseteq N_{C(T)}[u]$. We claim that $N_{2}\subseteq N_{C(T)}[u]$. Indeed, let $w\in N_{2}$. Then $v$ lies in a branch of $T$ which starts at $p(w)$. Since $u$ is an ancestor of $v$, we deduce that either $u$ lies in the $(r,p(w))$-path in $T$, or $u$ lies in the $(p(w),v)$-path in $T$. In both cases we have that $u$ is adjacent with $w$ in $C(T)$. Hence, $N_{2}\subseteq N_{C(T)}[u]$. We claim that $N_{3}\subseteq N_{C(T)}[u]$. Indeed, let $w\in N_{3}$. Then $w\in c(v)$. Since $u$ is an ancestor of $v$, we have that $u$ is either equal to $p(v)$ or an ancestor of $p(v)$, and thus $w$ is a descendant of $u$. Hence, $w \in N_{C(T)}[u]$. By the above it follows that $N_{C(T)}[v] \subseteq N_{C(T)}[u]$. This completes the proof of \autoref{lem:Burling.collinear.vertices.are.compatible.in.C(T)}.
\end{proof}

Let $(T,r,l,c)$ be a Burling tree. A \defin{left principal branch} is a principal branch $B$ such that for every $u\in V(T)\setminus V(B)$ we have $c(u)\cap V(B)=\emptyset$.

\begin{lemma}
\label{lem:Burling.c(T).left.pricipal.branch}
    Let $(T,r,l,c)$ be a Burling tree. Then $T$ has a left principal branch.
\end{lemma}

\begin{proof}[Proof of \autoref{lem:Burling.c(T).left.pricipal.branch}]
    We construct a left principal branch of $T$ inductively as follows: Let $v_{0}:=r$. Let $i\geq 0$ and suppose that we have constructed a path $v_{0}, \ldots, v_{i}$. If $v_{i}$ is a leaf, then we are done. Otherwise, $v_{i}$ has at least one child. If $v_{i}$ has at least two children, then we let $v_{i+1}$ be a non-last-born of $v_{i}$, otherwise we let $v_{i+1}:=l(v_{i})$. 
    
    Let $B:= v_{0}, \ldots, v_{k}$ be a principal branch that has been created by the above process. We claim that $B$ is a left principal branch. Let us suppose towards a contradiction that there exists $u\in V(T)\setminus V(B)$ and $v\in V(B)$, such that $v \in c(u)$. Then $l(p(u)) \in V(B)$, but $p(u)$ has a non-last born child, namely $u$; this contradicts the construction of $B$. This completes the proof of \autoref{lem:Burling.c(T).left.pricipal.branch}.
\end{proof}

Let $(T,r,l,c)$ be a Burling tree, and let $G$ be a graph which is derived from $(T,r,l,c)$. We may assume that all leaves of $T$ are in $G$. We call a vertex $v$ of $G$ a \defin{bottom-left} vertex of $G$ if $v$ lies in a left principal branch $B$ of $T$, and $v$ is a leaf of $T$.

\begin{observation}
\label{obs:bottom-left.vertex.exists}
   Let $G$ be a derived graph. Then $G$ has at least one bottom-left vertex.
\end{observation}

\begin{lemma}
\label{lem:burling.is.strongly.chordal}
   The class $\mathcal{C}$ is a subclass of the class of strongly chordal graphs.
\end{lemma}

\begin{proof}[Proof of \autoref{lem:burling.is.strongly.chordal}]
We prove that for every Burling tree $(T,r,l,c)$, the graph $C(T)$ is strongly chordal. Let $(T,r,l,c)$ be a Burling tree. In order to prove that $C(T)$ is strongly chordal, by \autoref{thm:strongly.chordal.characterization.simple}, it suffices to prove that every induced subgraph of $C(T)$ has a simple vertex. Let $H$ be an induced subgraph of $C(T)$. By \autoref{obs:bottom-left.vertex.exists} we know that $G(T)[V(H)]$ has at least one bottom-left vertex. Let $v$ be such a vertex.

We claim that $v$ is a simple vertex of $C(T)$ and thus a simple vertex of $H$. By \autoref{lem:Burling.collinear.vertices.are.compatible.in.C(T)}, in order to prove that $v$ is a simple vertex in $H$, it suffices to prove that the neighborhood of $v$ in $C(T)$ is included in a branch of $T$. Indeed, since $v$ is a bottom-left vertex, we have that $v$ has no in-neighbors in $A(T)[V(H)]$ and no descendant of $v$ is in $A(T)[V(H)]$. Hence, the neighborhood of $v$ in $H$ is included in the principal branch of $T$ which contains the set $c(v)$ in the case that $v$ is not a last-born, and in the principal branch of $T$ which contains $v$ otherwise. This completes the proof of \autoref{lem:burling.is.strongly.chordal}.
\end{proof}

\begin{figure}
    \centering
    \includegraphics[scale=1.5, page=2]{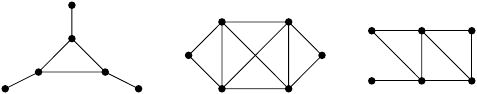}
    \caption{From left to right: The net, the complement of the graph $H$, and the complement of the graph $X_{96}$. Each of these graphs is a strongly chordal graph which is not a Burling strongly chordal graph.}
    \label{fig:forbidden.induced.subgraphs.of.C}
\end{figure}

We say that two vertices of a tree are \defin{collinear} if they both lie in the same principal branch. Since, by \autoref{thm:Burling}, Burling graphs (and thus derived graphs) are triangle-free, we have the following useful observation:

\begin{observation}
    \label{obs:chordal.triangle.lemma}
    Let $(T,r,l,c)$ be a Burling tree and let $a,b,c\in V(T)$ be such that the graph $C(T)[\{a,b,c\}]$ is a triangle. Then at least two of the vertices $a,b,c$ are collinear in $T$.
\end{observation}

The \defin{net} is the leftmost graph in \autoref{fig:forbidden.induced.subgraphs.of.C}.

\begin{lemma}
\label{lem:net.is.not.in.C}
    Every graph in the class $\mathcal{C}$ is net-free.
\end{lemma}

\begin{proof}[Proof of \autoref{lem:net.is.not.in.C}]
    Suppose not. Then, since $\mathcal{C}$ is hereditary, it follows that $\mathcal{C}$ contains the net. Let $G\in \mathcal{C}$ be isomorphic to the net, and let $(T,r,l,c)$ be a Burling tree such that $G$ is an induced subgraph of $C(T)$. By \autoref{obs:chordal.triangle.lemma}, we have that at least two vertices of the triangle of $G$, say $a$ and $b$, are collinear. Thus, by \autoref{lem:Burling.collinear.vertices.are.compatible.in.C(T)}, we have that $a$ and $b$ are compatible, which is a contradiction. This completes the proof of \autoref{lem:net.is.not.in.C}.
\end{proof}

Following the notation of the website \url{https://www.graphclasses.org} we denote by $H^c$ and $X_{96}^c$, the second and the third (from left to right) graph which is illustrated in \autoref{fig:forbidden.induced.subgraphs.of.C} respectively. Both $H^c$ and $X_{96}^c$ are strongly chordal graphs. We note that using \autoref{lem:Burling.collinear.vertices.are.compatible.in.C(T)} and \autoref{obs:chordal.triangle.lemma}, one can prove that every graph in $\mathcal{C}$ is $\{H^c, X_{96}^c\}$-free.

The following is an immediate corollary of \autoref{lem:net.is.not.in.C} and the fact that the net is a strongly chordal graph:

\begin{corollary}
\label{cor:C.proper.subclass.of.strongly.chordal}
        The class of Burling strongly chordal graphs is a proper subclass of the class of strongly chordal graphs.
\end{corollary}

In \autoref{subsec:finite.sets.conditions} we make use of the the following observation:

\begin{proposition}
\label{prop:trivially.perrfect.subclass.of.Burling.chordal}
    The class of trivially perfect graphs is a subclass of the class of Burling strongly chordal graphs.
\end{proposition}

Let $G_{1},\ldots, G_{k}$ be graphs. The \defin{disjoint union} of $G_{1},\ldots, G_{k}$, which we denote by $G_{1}+\ldots+G_{k}$, is the graph which has as vertex set (respectively edge set) the disjoint union of the sets $V(G_{1}), \ldots, V(G_{k})$ (respectively of the sets $E(G_{1}), \ldots, E(G_{k})$). In order to prove \autoref{prop:trivially.perrfect.subclass.of.Burling.chordal} we use the following characterization of trivially perfect graphs.

\begin{theorem}[Yan, Chen, and Chang {\cite[Theorem 3]{yan1996quasi}}]
\label{thm:trivially.perfect.characterization.operations}
    The class of trivially perfect graphs is the minimal hereditary class of graphs which contains the graph $K_{1}$, and is closed under the following operations:
    \begin{enumerate}[label=(\roman*)]
        \item disjoint union of two graphs;
        \item adding a new vertex complete to every other vertex.
    \end{enumerate}
\end{theorem}

Before we proceed to the proof of \autoref{prop:trivially.perrfect.subclass.of.Burling.chordal} we need to introduce some notation about functions. For a function $f$ we denote the domain of $f$ by $\dom(f)$. Let $f$ and $g$ be two functions which agree in the set $\dom(f)\cap \dom(g)$, that is for every $x\in \dom(f)\cap \dom(g)$ we have $f(x)=g(x)$. Then we denote by \defin{$f\cup g$} the function with $\dom(f\cup g)=\dom(f)\cup \dom(g)$, which is defined as follows: 
    \[ (f\cup g)(x) =\begin{cases} 
      f(x), & \text{if } x\in \dom(f);\\
      g(x), & \text{otherwise.}
   \end{cases}
   \]

Let $k$ be a positive integer, let $f$ be a function, and let $x_{1}, \ldots x_{k}$ be elements such that for every $i\in [k]$ we have that $x_{i}\notin \dom(f)$; we denote by $f\cup \{(x_{1},y_{1}), \ldots, (x_{k},y_{k})\}$ the function that we obtain by extending the definition of $f$ to include $f(x_{i})=y_{i}$ for every $i\in [k]$. We are now ready to proceed with the proof of \autoref{prop:trivially.perrfect.subclass.of.Burling.chordal}.

\begin{proof}[Proof of \autoref{prop:trivially.perrfect.subclass.of.Burling.chordal}]
    Since $\mathcal{C}$ contains $K_{1}$, it suffices to show that $\mathcal{C}$ is closed under the two operations which are mentioned in the statement of \autoref{thm:trivially.perfect.characterization.operations}. Let $(T_1,r_1,l_1,c_1)$ and $(T_2,r_2,l_2,c_2)$ be Burling trees on disjoint vertex sets.

    We claim that $C(T_1)+C(T_2)\in \mathcal{C}$. Indeed, let $T$ be the tree which we obtain from $T_1 + T_2$ by adding the three new vertices $r,r_1'$ and $r_2'$, and the edges $r_1r_1', r_2r_2', rr_1'$ and $rr_2'$. Let $l:= l_1\cup l_2\cup \{(r,r_2'), (r_1',r_1), (r_2',r_2)\}$ and $c:=c_1\cup c_2\cup\{(r_1',B), (r,\emptyset),(r_2',\emptyset)\}$ where $B$ is a branch in $T$ starting at $r_2'$. Consider the Burling tree $(T,r,l,c)$. Then $C(T_1)+C(T_2)$ is the induced subgraph of $C(T)$ that we obtain by deleting the vertices $r,r_1'$ and $r_2'$. This proves that $\mathcal{C}$ is closed under disjoint union.

    Let $C'(T_1)$ be the graph that we obtain from $C(T_1)$ by adding a new vertex, say $r$, which is complete to $V(C(T_1))$. Let $T$ be the tree which we obtain from $T_1$ by adding the new vertex $r$ and the edge $rr_1$. Consider the Burling tree $(T,r,l_1\cup\{(r,r_1)\},c_1)$. Then $C'(T_1)$ is isomorphic to $C(T)$, and thus $C'(T_1)\in \mathcal{C}$. This proves that $\mathcal{C}$ is closed under operation $(ii)$ in the statement of $\autoref{thm:trivially.perfect.characterization.operations}$. This completes the proof of \autoref{prop:trivially.perrfect.subclass.of.Burling.chordal}.
\end{proof}

We were not able to fully characterize Burling strongly chordal graphs. We suggest the following:

\begin{problem}
   Characterize the class of Burling strongly chordal graphs by its forbidden induced subgraphs.
\end{problem}

Let $G$ be an oriented graph which is derived from a Burling tree $(T,r,l,c)$. Following Pournajafi and Trotignon \cite{pournajafi2023burling} we call a vertex $v$ of $G$ a \defin{top-left} vertex if the following hold:
\begin{itemize}[topsep=0.2cm,itemsep=-0.4cm,partopsep=0.5cm,parsep=0.5cm]
    \item the distance of $v$ from $r$ in $T$ is equal to the minimum distance of a vertex of $G$ from $r$ in $T$; and
    \item either $v$ is not a last-born of $T$ or $v$ is the only vertex in $\argmin\{d_{T}(u,r):u\in V(G)\}$.
\end{itemize}

\begin{lemma}[Pournajafi and Trotignon {\cite[Lemma 3.1]{pournajafi2023burling}}]
\label{lem:Burling.top-left}
    Every non-empty oriented graph $G$ derived from a Burling tree $(T,r,l,c)$ contains at least one top-left vertex and every such vertex is a source of $G$. Moreover, the neighborhood of a top-left vertex is included in a branch of $T$.
\end{lemma}

A \defin{universal vertex} of a graph $G$ is a vertex $v\in V(G)$ which is complete to $V(G)\setminus \{v\}$. We need the following characterization of trivially perfect graphs.

\begin{theorem}[Wolk \cite{wolk1962comparability}]
\label{thm:characterization.trivially.perfect}
    Let $G$ be a graph. Then $G$ is trivially perfect if and only if every connected induced subgraph of $G$ contains a universal vertex.
\end{theorem}

We recall the definition of the class $\mathcal{I}$: Let $(T,r,l,c)$ be a Burling tree. We denote by $I(T)$ the graph that we obtain from $G(T)$ by adding edges so that every vertex is adjacent to all of its siblings, and to all the descendants of its last-born sibling. Then $\mathcal{I}$, as we defined it earlier in this section, is the closure under induced subgraphs of the class $\{I(T):(T,r,l,c) \text{ is a Burling tree}\}$.

\begin{lemma}
\label{lem:interval.derived.trivially.perfect}
    The class $\mathcal{I}$ is a subclass of the class of trivially perfect graphs.
\end{lemma}

\begin{proof}[Proof of \autoref{lem:interval.derived.trivially.perfect}]
    We prove that for every Burling tree $(T,r,l,c)$, the graph $I(T)$ is trivially perfect. Let $(T,r,l,c)$ be a Burling tree. By \autoref{thm:characterization.trivially.perfect} in order to prove that $I(T)$ is trivially perfect it suffices to prove that every connected induced subgraph of $I(T)$ has a universal vertex. 
    
    Let $H$ be a connected induced subgraph of $I(T)$. By \autoref{lem:Burling.top-left} we know that $A(T)[V(H)]$ has a top-left vertex. Let $v$ be such a vertex. We claim that $v$ is a universal vertex of $H$. Suppose not. Let $u\in V(H)$  be a vertex which is not adjacent to $v$. We claim that the vertex set of each connected component of $I(T)$ consists of a set $S$ of siblings in $T$ and the set $D$ of the descendants of the last-born, say $l$, vertex in $S$; where $S$ is a clique, and every vertex of $S$ is complete to $D$. Indeed, our claim follows by the definition of $I(T)$ and the fact that the vertex set of each branch of $T$ is an independent set in $G(T)$. Now since $u$ and $v$ are non-adjacent in $H$, it follows that $u,v\in D$; that is, both $u$ and $v$ are descendants of $l$ in $T$. Thus, we have that $v$ is not a top-left vertex, which is a contradiction. This completes the proof of \autoref{lem:interval.derived.trivially.perfect}.
\end{proof}

\begin{lemma}
\label{lem:trivially.perfect.in.I}
    The class of trivially perfect graphs is a subclass of the class $\mathcal{I}$.
\end{lemma}

\begin{proof}[Proof of \autoref{lem:trivially.perfect.in.I}]
    As in the proof of \autoref{prop:trivially.perrfect.subclass.of.Burling.chordal} we use the characterization of trivially perfect graphs from \autoref{thm:trivially.perfect.characterization.operations}. Since $\mathcal{I}$ contains $K_{1}$, it suffices to show that $\mathcal{I}$ is closed under the two operations which are mentioned in the statement of \autoref{thm:trivially.perfect.characterization.operations}. Let $(T_1,r_1,l_1,c_1)$ and $(T_2,r_2,l_2,c_2)$ be Burling trees on disjoint vertex sets.

    We claim that $I(T_1)+I(T_2)\in \mathcal{C}$. Indeed, let $T$ be the tree which we obtain from $T_1 + T_2$ by adding the three new vertices $r,r_1'$ and $r_2'$, and the edges $r_1r_1', r_2r_2', rr_1'$ and $rr_2'$. Let $l:=l_1\cup l_2\cup \{(r,r_2'), (r_1',r_1), (r_2',r_2)\}$ and $c:=c_1\cup c_2\cup\{(r_1',B), (r,\emptyset),(r_2',\emptyset)\}$ where $B$ is a branch in $T$ starting at $r_2'$. Consider the Burling tree $(T,r,l,c)$. Then $I(T_1)+I(T_2)$ is the induced subgraph of $I(T)$ that we obtain by deleting the vertices $r,r_1'$ and $r_2'$. This proves that $\mathcal{I}$ is closed under disjoint union.

    Let $I'(T_1)$ be the graph that we obtain from $I(T_1)$ by adding a new vertex, say $r'$, which is complete to $V(I(T_1))$. Let $T$ be the tree which we obtain from $T_1$ by adding the two new vertices $r,r'$ and the edges $rr'$ and $rr_1$. Let $l:=l_1\cup\{(r,r_1)\}$, and let $c:=c_{1}\cup \{(r,\emptyset), (r',\emptyset)\}$. Consider the Burling tree $(T,r,l,c)$. Then $I'(T_1)$ is isomorphic to the graph that we obtain from $I(T)$ by deleting $r$, and thus $I'(T_1)\in \mathcal{C}$. This proves that $\mathcal{I}$ is closed under operation $(ii)$ in the statement of $\autoref{thm:trivially.perfect.characterization.operations}$. This completes the proof of \autoref{lem:trivially.perfect.in.I}.
\end{proof}

The following is an immediate corollary of \autoref{lem:interval.derived.trivially.perfect} and \autoref{lem:trivially.perfect.in.I}:

\begin{corollary}
\label{cor:I.equals.trivially.perfect}
    The class $\mathcal{I}$ is the class of trivially perfect graphs.
\end{corollary}

Now \autoref{thm:Burling.derived.in.intersection.} is an immediate corollary of \autoref{obs:derived.is.equal.to.intersection}, \autoref{cor:C.proper.subclass.of.strongly.chordal}, and \autoref{cor:I.equals.trivially.perfect}. Recall that \autoref{thm:Burling.derived.in.intersection.} implies that both the class of trivially perfect graphs and Burling strongly perfect graphs are not \g{}. Hence if $\mathcal{H}$ is a finite set of graphs such that the class of $\mathcal{H}$-free graphs is \g{}, then $\mathcal{H}$ contains a Burling strongly chordal graph and a trivially perfect graph. Since, by \autoref{prop:trivially.perrfect.subclass.of.Burling.chordal}, every trivially perfect graph is a Burling strongly chordal graph, the condition for $\mathcal{H}$ to contain both a Burling strongly chordal graph and a trivially perfect graph is satisfied when $\mathcal{H}$ contains a trivially perfect graph. We summarize these in the following corollary:

\begin{corollary}
\label{cor:H.finite.guarding.trivially.perfect}
    Let $\mathcal{H}$ be a set of graphs. If the class of $\mathcal{H}$-free graphs is \g{}, then $\mathcal{H}$ contains a trivially perfect graph.
\end{corollary}

\subsection{Necessary conditions a finite set of graphs should satisfy in order to define a \texorpdfstring{\g}{TEXT} class}
\label{subsec:finite.sets.conditions}

We need the following classic result of Erd\H{o}s:

\begin{theorem}[Erd\H{os} \cite{erdos1959graph}]
    \label{thm:Erdos.large.girth.large.chi}
    Let $g\geq 3$ and $k\geq 2$. Then there exists a graph of girth at least $g$ and chromatic number at least $k$.
\end{theorem}

The following is an immediate corollary of \autoref{thm:Erdos.large.girth.large.chi}:

\begin{observation}
\label{obs:finite.chib.contains.forest}
    Let $\mathcal{H}$ be a finite set of graphs. If the class of $\mathcal{H}$-free graphs is $\chi$-bounded, then $\mathcal{H}$ contains a forest.
\end{observation}

The following is an immediate corollary of \autoref{obs:chig.is.chib} and \autoref{obs:finite.chib.contains.forest}:

\begin{corollary}
\label{cor:H.finite.guarding.forest}
    Let $\mathcal{H}$ be a finite set of graphs. If the class of $\mathcal{H}$-free graphs is \g{}, then $\mathcal{H}$ contains a forest.
\end{corollary}

\autoref{cor:H.finite.guarding.complete.multipartite}, \autoref{cor:H.finite.guarding.line.of.forest}, \autoref{cor:H.finite.guarding.trivially.perfect}, and \autoref{cor:H.finite.guarding.forest} give rise to necessary conditions on a finite set $\mathcal{H}$ of graphs for the class of $\mathcal{H}$-free graphs to be \g{}. We summarize these conditions in the following theorem:

\begin{theorem}
\label{thm:finite.chi-guarding.necessity}
    Let $\mathcal{H}$ be a finite set of graphs. If the class of $\mathcal{H}$-free graphs is \g{}, then $\mathcal{H}$ contains:
    \begin{itemize}[topsep=0.2cm,itemsep=-0.4cm,partopsep=0.5cm,parsep=0.5cm]
        \item a forest;
        \item a line graph of a forest;
        \item a complete multipartite graph; and
        \item a trivially perfect graph;
    \end{itemize}
\end{theorem}

We note that in an upcoming paper we discuss results towards a complete characterization of the \g{} classes which are defined by two forbidden induced subgraphs.

\section{A characterization of \texorpdfstring{$H$}{TEXT}-free intersectionwise \texorpdfstring{$\chi$}{TEXT}-guarding classes of graphs}
\label{sec:H-free}

In this section we characterize the graphs $H$ for which the class of $H$-free graphs is \g{}. The main result of this section is the following:

\begin{theorem}
\label{thm:characterization.H-free.guarding}
    Let $H$ be a graph. Then the class of $H$-free graphs is \g{} if and only if $H$ is isomorphic to $P_2$, $P_3$, or $rK_1$ for some $r > 0$.
\end{theorem}

We start by proving the backwards direction of \autoref{thm:characterization.H-free.guarding}. Recall that in \autoref{sec:decomposable.graphs.guarding} we proved \autoref{lem:rK1-free.guarding} which states that for every positive integer $r$, the class of $rK_{1}$-free graphs is \g{}. Thus, since every $P_{2}$-free graph is a $P_{3}$-free graph, in order to prove the backwards direction it suffices to prove that the class of $P_{3}$-free graphs is \g{}. To this end we need the following well-known result:

\begin{proposition}[Folklore]
\label{prop:P3-free_structure}
    Let $G$ be a graph. Then $G$ is the disjoint union of complete graphs if and only if $G$ is $P_{3}$-free.
\end{proposition}

In \autoref{sec:decomposable.graphs.guarding} we proved \autoref{cor:componentwise.r.dependent.guarding} which states that every class of componentwise $r$-dependent graphs is \g{}. Since, by \autoref{prop:P3-free_structure}, every $P_{3}$-free graph is componentwise $2$-dependent, we get the following:

\begin{corollary}
\label{cor:P3-free.guarding}
   The class of $P_{3}$-free graphs is \g{}.
\end{corollary}

Since a graph is complete multipartite if and only if its complement is the disjoint union of complete graphs, \autoref{prop:P3-free_structure} implies the following:

\begin{proposition}[Folklore]
\label{prop:charact.of.compl.multi.induced}
    Let $G$ be a graph. Then $G$ is complete multipartite if and only if $G$ is $(K_{1}+K_{2})$-free.
\end{proposition}

We now proceed with proving the characterization of \g{} graph classes which are defined by a single forbidden induced subgraph, as stated in \autoref{thm:characterization.H-free.guarding}.

\begin{proof}[Proof of \autoref{thm:characterization.H-free.guarding}]
    The backward direction follows directly from \autoref{cor:P3-free.guarding}, \autoref{lem:rK1-free.guarding}, and the fact that $P_2$ is an induced subgraph of $P_3$.
    
    For the forward direction: Let $H$ be a graph such that the class of $H$-free graphs is \g{}. We first notice that, by \autoref{cor:H.finite.guarding.forest}, $H$ is a forest. 

    Next, we observe that $H$ has to be claw-free, as otherwise we would obtain a contradiction with \autoref{cor:line.graph.guarding}. Hence, $H$ is a disjoint union of paths. If $H$ consists of multiple components, then each of these is a single vertex, since otherwise $H$ contains $K_{1}+K_{2}$, which contradicts the fact that, by \autoref{thm:finite.chi-guarding.necessity}, $H$ is complete multipartite, and thus, by \autoref{prop:charact.of.compl.multi.induced}, $H$ is $(K_{1}+K_{2})$-free. In this case, $H$ is isomorphic to $rK_1$ for some $r \geq 2$.

    Hence, we may assume that $H$ is a path. $H$ does not contain more than $3$ vertices, as any path on at least $4$ vertices contains $K_{1}+K_{2}$ as an induced subgraph, leading to a contradiction with \autoref{thm:finite.chi-guarding.necessity}. Thus, $H$ is $P_1$ (which equals $rK_1$ for $r = 1$), $P_2$, or $P_3$, each of which has previously been shown to be \g{}. This completes the proof of \autoref{thm:characterization.H-free.guarding}.
\end{proof}

\section{Intersectionwise self-\texorpdfstring{$\chi$}{TEXT}-guarding classes}
\label{sec:self.chi}

 As we discussed in \autoref{sec:intro} we call a class of graphs $\mathcal{A}$ \sg{} if for every positive integer $k$ the class $\gcap_{i\in [k]}\mathcal{A}$ is $\chib$. In this section we study \sg{} classes. In \autoref{sub:sxg.one.forbidden.is} we focus on \sg{} classes which are defined by one forbidden induced subgraph, and in \autoref{subsec:sxg.from.chig} we consider ways to construct \sg{} classes from \g{} classes.

\subsection{\texorpdfstring{$H$}{TEXT}-free intersectionwise self-\texorpdfstring{$\chi$}{TEXT}-guarding classes}
\label{sub:sxg.one.forbidden.is}

We begin with an easy observation. Since for every integer $k$, the $k$-fold graph-intersection of a class $\mathcal{C}$ contains $\mathcal{C}$, we have the following:

\begin{observation}
\label{obs:sg.implies.chi.bounded}
    Let $\mathcal{C}$ be a \sg{} class of graphs. Then $\mathcal{C}$ is $\chi$-bounded.
\end{observation}

We remark that the converse of \autoref{obs:sg.implies.chi.bounded} does not hold. In particular, as we mention in \autoref{sec:intro}, it follows from a result of Burling \cite{burling1965coloring} that the class of interval graphs is not \sg{}, and thus the class of perfect graphs is not \sg{}. The following is an immediate corollary of \autoref{obs:sg.implies.chi.bounded} and \autoref{obs:finite.chib.contains.forest}.

\begin{corollary}
\label{cor:H.sfg.is.forest}
    Let $H$ be a graph. If the class of $H$-free graphs is \sg{}, then $H$ is a forest.
\end{corollary}

A \defin{chair} is a graph isomorphic to the graph on five vertices which we obtain by identifying one of the vertices of a $P_{2}$ with a vertex of degree two of a $P_{4}$. 

\begin{theorem}
\label{thm:sg.chair.claw.plus}
    The following classes of graphs are not \sg:
    \begin{enumerate}[label=(\roman*)]
        \item the class of chair-free graphs.
        \item the class of $(K_{1,3}+K_{1})$-free graphs.
    \end{enumerate}
\end{theorem}

\begin{proof}[Proof of \autoref{thm:sg.chair.claw.plus}]
    Let $\mathcal{C}$ be the class of complete multipartite graphs, and let $\mathcal{D}$ the class of line graphs. Then, by \autoref{lem:intersection.line.graphs.complete.multipartite.not.chib}, the class $\mathcal{C} \gcap \mathcal{D}$ is not $\chi$-bounded.

    The theorem now follows by the observation that both the class of $(K_{1,3} + K_{1})$-free graphs and the class of chair-free graphs contain the class of line graphs (since line graphs are claw-free graphs) and the class of complete multipartite graphs (since complete multipartite graphs are $(K_{1}+ K_{2})$-free graphs). This completes the proof of \autoref{thm:sg.chair.claw.plus}.
\end{proof}

We say that a graph is a \defin{linear forest} if it is a forest in which every component is a path.

\begin{corollary}
\label{cor:H.sfg.linear.forest.OR.star}
    Let $H$ be a graph. If the class of $H$-free graphs is \sg{}, then $H$ is a linear forest or a star.
\end{corollary}

\begin{proof}[Proof of \autoref{cor:H.sfg.linear.forest.OR.star}]
    Let $H$ be a graph as in the statement. By \autoref{cor:H.sfg.is.forest} we have that $H$ is a forest. Suppose that $H$ is not a linear forest. Then a component, say $H_{1}$, of $H$ contains a claw. Since, by \autoref{thm:sg.chair.claw.plus}, we have that $H$ is $(K_{1,3}+K_{1})$-free, it follows that $H$ is connected. We claim that $H_{1}$ is a star. Suppose not. Then $H_{1}$ contains a chair, which contradicts \autoref{thm:sg.chair.claw.plus}. This completes the proof of \autoref{cor:H.sfg.linear.forest.OR.star}.
\end{proof}

Recently Adenwalla, Braunfeld, Sylvester, and Zamaraev \cite{adenwalla2024boolean}, using different terminology, proved that if $H$ is a star then the class all $H$-free graphs is \sg{}.

\begin{theorem}[{Adenwalla, Braunfeld, Sylvester, and Zamaraev \cite[Lemma 5.9]{adenwalla2024boolean}}]
\label{thm:self_chi_star}
    Let $t$ be a positive integer. Then the class of $K_{1,t}$-free graphs is \sg{}.
\end{theorem}

 In what follows we discuss a result of Gy\'arf\'as \cite{gyarfas1985problems} which implies that the class of $P_{4}$-free graphs is \sg{}. We first need to introduce some terminology. An orientation of an undirected graph $G$ is a \defin{transitive orientation} if the adjacency relation of the resulted directed graph is transitive. A \defin{comparability graph} is a graph which has a transitive orientation.

\begin{theorem}[{Gy\'arf\'as, \cite[Proposition 5.8]{gyarfas1985problems}}]
\label{thm:gyarfas.comparability}
    The class of comparability graphs is \sg{}.
\end{theorem}

Jung \cite{jung1978class} proved that every $P_{4}$-free graph is a comparability graph. Hence, we have the following:

\begin{corollary}
\label{cor:self_chi_star}
    The class of $P_{4}$-free graphs is \sg{}.
\end{corollary}

 In view of the above we propose the following conjecture:

\begin{conjecture}
\label{conj:self_chi_one_forbidden_is}
    If $H$ is a linear forest, then the class of $H$-free graphs is \sg{}.
\end{conjecture}

By \autoref{cor:H.sfg.linear.forest.OR.star} and \autoref{thm:self_chi_star} it follows that an affirmative answer to \autoref{conj:self_chi_one_forbidden_is} would imply a complete characterization of the \sg{} graph classes which are defined by a single forbidden induced subgraph. We were not able to decide the following: 

\begin{problem}
    Is the class of $P_{5}$-free graphs \sg?
\end{problem}

\begin{problem}
    Is the class of $(P_{4} + P_{1})$-free graphs \sg?
\end{problem}

In \autoref{subsec:sxg.from.chig} we will show that the class of $(P_{3}+rP_{2})$-free graphs is \sg.

\subsection{Intersectionwise self-\texorpdfstring{$\chi$}{TEXT}-guarding classes from intersectionwise \texorpdfstring{$\chi$}{TEXT}-guarding classes}
\label{subsec:sxg.from.chig}

The main result of this subsection is the following theorem which allows us to construct new \sg{} classes from \g{} classes which are defined by a finite set of forbidden induced subgraphs:

\begin{theorem}
\label{thm:finite.chi.guarding.class.add.K2.self}
        Let $k$ and $t$ be a positive integers and let $\mathcal{H} = \{H_{1}, \ldots , H_{k}\}$ be a set of graphs.
        For every $i\in [t]$ let $r_{1}^{i}, \ldots, r_{k}^{i}$ be $k$ nonnegative integers and let $\mathcal{H}^{i} := \{H_{j} + r_{j}^{i}K_{2}:j\in [k]\}$. If the class of $\mathcal{H}$-free graphs is \g{}, then the class $\gcap_{i\in[t]}\{\mathcal{H}^{i}\text{-free graphs}\}$ is $\chi$-bounded. In particular, for every $i\in [t]$ the class of $\mathcal{H}^{i}$-free graphs is \sg{}.
\end{theorem}

An application of \autoref{thm:finite.chi.guarding.class.add.K2.self} is the following immediate corollary of \autoref{thm:characterization.H-free.guarding} and \autoref{thm:finite.chi.guarding.class.add.K2.self} which settles some cases of \autoref{conj:self_chi_one_forbidden_is}.

\begin{corollary}
    \label{cor:P_{3}+rK_{2}.is.self.chi}
    Let $r$ be a positive integer. Then the class of $(P_{3}+rK_{2})$-free graphs is \sg{}.
\end{corollary}

The following is an immediate corollary of \autoref{prop:charact.of.compl.multi.induced} which states that a graph $G$ is complete multipartite if and only if $G$ is $(P_{2}+K_{1})$-free, and \autoref{cor:P_{3}+rK_{2}.is.self.chi}:

\begin{corollary}
\label{cor:complete.multipartite.self.chi}
    The class of complete multipartite graphs is \sg{}.
\end{corollary}

We remark that Adenwalla, Braunfeld, Sylvester, and Zamaraev proved the following strengthening of \autoref{cor:complete.multipartite.self.chi}:

\begin{proposition}[{Adenwalla, Braunfeld, Sylvester, and Zamaraev \cite[Proposition 5.26]{adenwalla2024boolean}}]
  For every positive integer $k$ the $k$-fold graph-intersection of the class of complete multipartite graphs is $\chib$ by the linear function $f(x)=k^{2^{k}}x$.  
\end{proposition}

The rest of \autoref{subsec:sxg.from.chig} is devoted to the proof of \autoref{thm:finite.chi.guarding.class.add.K2.self}, for which we use a technique inspired by previous work of Chudnovsky, Scott, Seymour, and Spirkl \cite{pureI}: We assume towards a contradiction that there exists a graph $G=G_{1}\cap\ldots \cap G_{t}$ in the class $\gcap_{i\in[t]}\{\mathcal{H}^{i}\text{-free graphs}\}$ with chromatic number large in terms of its clique number. We then find in $G$ a large collection $\mathcal{B}$ of pairwise disjoint sets of vertices ("boxes") such that for each $B\in \mathcal{B}$ the subgraph which is induced by $B$ has large chromatic number. For a supergraph $H$ of $G$ we consider a pair $\{B,B'\}$ of boxes dense in the graph $H$ if for every $v\in B$ we have that $\chi(G[B' \setminus N_{H}(v)])$ is small, and for every $v\in B'$ we have that $\chi(G[B \setminus N_{H}(v)])$ is small. For each pair of boxes $\{B,B'\}\subseteq \mathcal{B}$, and each $i\in [t]$, we ask if the pair $\{B,B'\}$ could be made dense in $G_{i}$, possibly by "shrinking" (with respect to their chromatic number in $G$) the boxes $B$ and $B'$ by a constant factor. This results in a "shrink-resistant" configuration. Now, because we are applying induction on the clique number of $G$, and neighborhoods of vertices in $G$ have clique number smaller than $\omega(G)$, it follows that $G$ is relatively sparse. By playing this fact against the fact that each $G_{i}$ is relatively dense (for every edge, its common non-neighbors have small chromatic number by the inductive setup) we get a contradiction.

For the proof of \autoref{thm:finite.chi.guarding.class.add.K2.self} we need the notion of edge-coloring and a version of Ramsey's theorem. Let $k$ be a positive integer, and let $G$ be a graph. A \defin{$k$-edge-coloring} of $G$ is a function $f:E(G) \to[k]$.

\begin{theorem}[{Ramsey \cite[Theorem B]{ramsey1930problem}}]
    \label{thm:ramsey_edge_monochromatic_clique}
    Let $k$ and $t$ be positive integers. Then there exists an integer $n(k,t)$ such that if $G$ is a complete graph on $n(k,t)$ vertices, then every $k$-edge-coloring of $G$ results in a monochromatic clique of size $t$.
\end{theorem}

Let $k$ and $t$ be positive integers. We denote by $R_{k}(t)$ the minimum positive integer such that if $G$ is a complete graph on $R_{k}(t)$ vertices, then every $k$-edge-coloring of $G$ results in a monochromatic clique of size $t$. We are now ready to proceed with the proof of \autoref{thm:finite.chi.guarding.class.add.K2.self}.

\begin{proof}[Proof of \autoref{thm:finite.chi.guarding.class.add.K2.self}]
We prove the theorem by induction on $t$. For the base case, where $t=1$, we prove by induction on $r:=\sum_{j\in[k]}r_{j}^{1}$ that the class of $\mathcal{H}^{1}$-free graphs is $\chi$-bounded.

Suppose that $r=0$. Then $\mathcal{H}^{1}$ is the class $\mathcal{H}$. Thus, by our assumptions, 
the class of $\mathcal{H}^{1}$-free graphs is \g{}. Hence, by \autoref{obs:chig.is.chib}, we have that the class of $\mathcal{H}^{1}$-free graphs is $\chi$-bounded.

Let us suppose that $r>0$, and let $j\in[k]$ be such that $r_{j}^{1}>0$. From the induction hypothesis we have that the class of $\{H_{1}+r_{1}^{1}K_{2}, \ldots, H_{j}+(r_{j}^{1}-1)K_{2}, \ldots, H_{k}+r_{k}^{1}K_{2}\}$-free graphs is $\chib$. Let $f_{r-1}$ be a $\chi$-bounding function for this class. 

\begin{claim}
\label{cl:thm.K2.k=1}
    The class of $\mathcal{H}^{1}$-free graphs is $\chi$-bounded by the function $f_{r}:\mathbb{N}\to \mathbb{N}$ which is defined recursively as follows:

    \[ f_{r}(n) =\begin{cases} 
      1, & \text{if } n=1;\\
      2f_{r}(n -1) + f_{r-1}(n), &\text{otherwise.}
   \end{cases}
   \]
\end{claim}

\begin{subproof}[Proof of \autoref{cl:thm.K2.k=1}]
We prove the claim by induction on $\omega:=\omega(G)$. The claim holds trivially when $\omega=1$.
Suppose that $\omega>1$, and that the claim hols for graphs $H$ such that $\omega(H)<\omega$.

Let $G$ be an $\mathcal{H}^{1}$-free graph, and let $uv\in E(G)$. Then $\{N(u)\cup N(v), A(u)\cap A(v)\}$ is a partition of $V(G)$, and thus we have that $\chi(G) \leq \chi(G[N(u)\cup N(v)]) + \chi(G[A(u)\cap A(v)])$. Suppose for contradiction that the graph $G[A(u)\cap A(v)]$ contains $H_{j}+(r_{j}^{1}-1)K_{2}$. Then the graph $G[A(u)\cap A(v)\cup\{u,v\}]$ contains $H_{j}+r_{j}^{1}K_{2}$, which is a contradiction. Hence, we have that the graph $G[A(u)\cap A(v)]$ is $(H_{j}+(r_{j}^{1}-1)K_{2})$-free. Thus, by the induction hypothesis, we have that $\chi(G[A(u)\cap A(v)]) \leq f_{r-1}(\omega)$. Also, $\omega(G[N(v)])<\omega(G)$, and thus, by the induction hypothesis, we have that $\chi(G[N(v)])\leq f_{r}(\omega-1)$. Similarly, $\chi(G[N(u)])\leq f_{r}(\omega-1)$. Finally, putting the above together we have that $\chi(G) \leq \chi(G[N(u)\cup N(v)]) + \chi(G[A(u)\cap A(v)]) \leq 2f_{r}(\omega-1) + f_{r-1}(n)$. This completes the proof of \autoref{cl:thm.K2.k=1}.
\end{subproof}

This concludes the proof of the case $t=1$. Suppose that $t>1$, and that our statement holds for every positive integer $t'<t$. For each $i\in[t]$, let $\mathcal{A}_{i}:= \gcap_{j\in[t]\setminus \{i\}}\{\mathcal{H}^{j}\text{-free graphs}\}$. By the induction hypothesis, it follows that for every $i\in[t]$, the class $\mathcal{A}_{i}$ is $\chi$-bounded. Thus, since the class $\mathcal{H}$ is \g{}, we have that the class $\mathcal{H} \gcap \mathcal{A}_{i}$ is $\chi$-bounded. For each $i\in[t]$, let $g_{i}$ be a $\chi$-bounding function for the class $\mathcal{H} \gcap \mathcal{A}_{i}$. Let $f\colon \mathbb{N}_{+}^{t+1}\rightarrow \mathbb{N}_{+}$ be the function which is defined recursively as follows:
\[ f(n,n_{1}, \ldots, n_{t}) =\begin{cases} 
      1, \ \text{if } n=1;\\
      \min_{i\in [t]}\{g_{i}(n): n_{i}=0\}, \ \text{if } n\neq 1 \text{ and there exists } i\in[t] \text{ such that } n_{i}=0;\\
      M(n,n_{1}, \ldots, n_{t})\cdot (t+2)^{t\cdot (R_{2^{t}}(t+2)-1)}\cdot R_{2^{t}}(t+2), \ \text{otherwise.}
   \end{cases}
\]
where $$M(n,n_{1}, \ldots, n_{t}):=\max\{t^{t+1} + 1, (t+1)f(n-1,n_{1}, \ldots, n_{t})+1,C(n,n_{1}, \ldots, n_{t})(t+1)(t+2)\},$$
and $$C(n,n_{1}, \ldots, n_{t}):=\max\{f(n,n_{1}, \ldots,n_{i}-1, \ldots, n_{t}):i\in[t]\}.$$

Let $G\in \gcap_{i\in[t]}\{\mathcal{H}^{i}\text{-free graphs}\}$ and for every $i\in [t]$ let $G_{i}$ be an $\mathcal{H}^{i}\text{-free}$ graph such that $G=\cap_{i\in[t]}G_{i}$. For each $i\in [t]$ let $r_{i}:= \sum_{j\in[k]} r_{j}^{i}$. In what follows we prove by induction on $\omega(G)$ that:
\begin{align}
\label{eq.thmK2.chi(G).main}
    \chi(G) \leq f(\omega(G), r_{1}, \ldots , r_{t}).
\end{align}
If $\omega(G)=1$, then (\ref{eq.thmK2.chi(G).main}) holds trivially. Let $\omega := \omega(G)>1$, and suppose that for every graph $H\in \gcap_{i\in[t]}\{\mathcal{H}^{i}\text{-free graphs}\}$ with $\omega(H)< \omega$, we have $\chi(H) \leq f(\omega(H), r_{1}, \ldots , r_{t})$. 

We prove, by induction on $\prod_{i\in[t]} r_{i}$, that $G$ satisfies (\ref{eq.thmK2.chi(G).main}). For the basis of the induction we observe that if there exists $i\in [t]$ such that $r_{i}=0$, then $G\in \mathcal{H} \gcap \mathcal{A}_{i}$, and thus $\chi(G)\leq \min_{i\in [t]}\{g_{i}(\omega): r_{i}=0\}$. In particular, $G$ satisfies (\ref{eq.thmK2.chi(G).main}) when $\prod_{i\in[t]} r_{i}=0$.

Let us suppose that for every $i\in [t]$, we have $r_{i}\geq 1$, and that for every graph $H\in \gcap_{i\in[t]}\{\mathcal{H}^{i}\text{-free graphs}\}$, with $\omega(H)\leq \omega$, and for every choice of $kt$ nonnegative integers $s_{j}^{i}$, where $i\in[t]$ and $j\in [k]$, such that $\prod_{i\in[t]} \left( \sum_{j\in[k]}s_{j}^{i}\right) < \prod_{i\in[t]} r_{i}$ we have: $$\chi(H)\leq f\left(\omega(H), \sum_{j\in[k]}s_{j}^{1}, \ldots , \sum_{j\in[k]}s_{j}^{t}\right).$$

Let us suppose towards a contradiction that $\chi(G)>f(\omega, r_{1}, \ldots , r_{t})$. For each $i\in[t]$, we denote by $c_{i}$ the number $f(\omega,r_{1}, \ldots,r_{i}-1, \ldots, r_{t})$. Then we have that $C(\omega, r_{1}, \ldots , r_{t})=\max\{c_{i}:i\in[t]\}$. We denote $C(\omega, r_{1}, \ldots , r_{t})$ by $C$. We denote $M(\omega, r_{1}, \ldots , r_{t})$ by $M$. We need to introduce a few new notions: Let $X$ and $Y$ be disjoint subsets of $V(G)$, and let $i\in[t]$. We say that $\{X,Y\}$ is an \defin{$i$-dense pair} in $G$ if the following hold:
\begin{itemize}[topsep=0.2cm,itemsep=-0.4cm,partopsep=0.5cm,parsep=0.5cm]
    \item For all $x\in X$ we have that $\chi(G[Y\setminus N_{G_{i}}(x)])\leq c_{i}$; and
    \item For all $y\in Y$ we have that $\chi(G[X\setminus N_{G_{i}}(y)])\leq c_{i}$.
\end{itemize}

Let $\mathcal{X}$ be a family of disjoint subsets of $V(G)$, let $X, Y\in \mathcal{X}$, and let $i\in[t]$ be such that:
\begin{itemize}
    \item $\{X,Y\}$ is not an $i$-dense pair;
    \item There exist $X'\subseteq X$ and $Y'\subseteq Y$, such that: 
    \begin{itemize}[topsep=0.2cm,itemsep=-0.4cm,partopsep=0.5cm,parsep=0.5cm]
        \item $\chi(G[X'])\geq \frac{1}{t+2} \chi(G[X])$,  $\chi(G[Y'])\geq \frac{1}{t+2} \chi(G[Y])$; and 
        \item $\{X',Y'\}$ is an $i$-dense pair in $G$.
    \end{itemize}
\end{itemize}
Then we say that the pair $\{X,Y\}$ is \defin{$i$-shrinkable} and we refer to the family $(\mathcal{X}\setminus \{X,Y\})\cup \{X', Y'\}$ as the family which is obtained from $\mathcal{X}$ by \defin{$i$-shrinking the pair $\{X,Y\}$ to the pair $\{X', Y'\}$}. If a family $\mathcal{X}$ of disjoint subsets of $V(G)$ contains an $i$-shrinkable pair for some $i\in [t]$, we say that $\mathcal{X}$ is \defin{shrinkable}; otherwise we say that $\mathcal{X}$ is \defin{unshrinkable}. Finally, in what follows, for disjoint subsets $X$ and $Y$ of $V(G)$ we denote by $I(X,Y)$ the set $\{i\in [t]: \{X,Y\}\text{ is an } i\text{-dense pair in } G\}$. 

\begin{claim}
\label{cl:unshrinkable.family}
There exists an unshrinkable family $\mathcal{B}$ of pairwise disjoint subsets of $V(G)$ such that the following hold:
\begin{itemize}[topsep=0.2cm,itemsep=-0.4cm,partopsep=0.5cm,parsep=0.5cm]
    \item $\mathcal{B}$ has size $t+2$;
    \item for every $B\in \mathcal{B}$, we have that $\chi(B)=M$; and
    \item there exists a set $I\subseteq [t]$ such that for each pair $\{B,B'\}$ of distinct subsets in $\mathcal{B}$, we have that $I(B,B')=I$.
\end{itemize}
\end{claim}

\begin{subproof}[Proof of \autoref{cl:unshrinkable.family}]
Let $\mathcal{B} = \{B_{1}, \ldots ,B_{R_{2^{t}}(t+2)}\}$ be a partition of $V(G)$ which has size $R_{2^{t}}(t+2)$, and such that for every $B\in \mathcal{B}$, we have $\chi(B)\geq M\cdot (t+2)^{t\cdot (R_{2^{t}}(t+2)-1)}$. We observe that, since $\chi(G)>f(\omega, r_{1}, \ldots , r_{t})$, such a partition exists.

Let $\mathcal{B}'$ be the family of disjoint subsets of $V(G)$ which we obtain from $\mathcal{B}$ as follows: We start with $\mathcal{B}':=\mathcal{B}$. For every $i\in [t]$, and for all distinct $l,l'\in [R_{2^{t}}(t+2)]$, if there exist $A_{l}\subseteq B_{l}\in \mathcal{B}'$ and $ A_{l'}\subseteq B_{l'} \in \mathcal{B}'$ such that the pair $\{B_{l}, B_{l'}\}$ is $i$-shrinkable to the pair $\{A_{l}, A_{l'}\}$, then we let $\mathcal{B}'$ be the family which is obtained from $\mathcal{B'}$ by $i$-shrinking the pair $\{B_{l}, B_{l'}\}$ to the pair $\{A_{l}, A_{l'}\}$. We repeat this process until $\mathcal{B}'$ is unshrinkable. Since every initial element of $\mathcal{B}$ will be shrunk at most $t\cdot (R_{2^{t}}(t+2)-1)$ times, it follows that for every $B'\in \mathcal{B}'$ we have $\chi(B') \geq M$. By restricting the elements of $\mathcal{B}'$ to appropriate subsets, we may assume that for every $B'\in \mathcal{B}'$, we have $\chi(B') = M$.

Let $H$ be the complete graph on $\mathcal{B}'$, and let $\phi: E(H) \to 2^{[t]}$ be a $2^{t}$-edge-coloring of $H$ which is defined as follows: $\phi(\{A,B\}) = I(A,B) \subseteq [t]$. Then, by \autoref{thm:ramsey_edge_monochromatic_clique}, and the fact that $|V(H)|=R_{2^{t}}(t+2)$, it follows that $H$ contains a monochromatic, with respect to $\phi$, clique of size $t+2$. Let $\mathcal{B''}\subseteq \mathcal{B'}$ be such a clique, and let $I$ be the color of its edges. Then $\mathcal{B''}$ satisfies our \autoref{cl:unshrinkable.family}. This completes the proof of \autoref{cl:unshrinkable.family}.
\end{subproof}
 
 Let $\mathcal{B} = \{B_{1}, \ldots , B_{t+2}\}$ be a family of disjoint subsets of $V(G)$ and let $I$ be a subset of $[t]$ as in the statement of \autoref{cl:unshrinkable.family}.

\begin{claim}
    \label{cl:pair.of.boxes.nonnbrs.large.chi}
    Let $B_{l}$ and $B_{l'}$ be distinct elements of $\mathcal{B}$. Then for every $x\in B_{l}$ there exists $i\in[t]\setminus I$ such that $\chi(G[B_{l'}\setminus N_{G_{i}}(x)])\geq \frac{1}{t+1}M$.
\end{claim}

\begin{subproof}[Proof of \autoref{cl:pair.of.boxes.nonnbrs.large.chi}]
Let us suppose towards a contradiction that there exists $x\in B_{l}$ such that for every $i\in [t]\setminus I$ we have that $\chi(G[A_{G_{i}[B_{l'}]}(x)]) < \frac{1}{t+1}M$. Let $i\in I$. Then, since $\{B_{l},B_{l'}\}$ is an $i$-dense pair, we have that $\chi\left(G[A_{G_{i}[B_{l'}]}(x)] \right) \leq c_{i}$. Thus, we have that: 

\begin{align*}
\chi\left(G[\cup_{i\in I}A_{G_{i}[B_{l'}]}(x)]\right) &\leq  \sum_{i\in I} \chi\left(G[A_{G_{i}[B_{l'}]}(x)] \right) \leq \sum_{i\in I}c_{i} \leq |I|C \\
& \leq |I|\frac{t}{t+1}C(t+1)(t+2) \leq |I|\frac{t}{t+1}M.
\end{align*}
Hence, we have:

\begin{align*}
    \chi\left(G[A_{G[B_{l'}]}(x)]\right) &= \chi\left(G\Big[\left(\cup_{i\in I}A_{G_{i}[B_{l'}]}(x)\right)\cup \left(\cup_{i\in [t]\setminus I}A_{G_{i}[B_{l'}]}(x)\right)\Big]\right)\\
    & \leq \chi\left(G[\cup_{i\in I}A_{G_{i}[B_{l'}]}(x)]\right) + \chi\left(G[\cup_{i\in [t]\setminus I}A_{G_{i}[B_{l'}]}(x)]\right)\\
    & < |I|\frac{t}{t+1}M+ (t-|I|)\frac{M}{t+1} = \frac{t}{t+1}M.
\end{align*}

Since, by the induction hypothesis, we have that $\chi(G[N_{G}(x)]) \leq f(\omega-1, r_{1}, \ldots , r_{t})$, it follows that:
\begin{align*}   
\chi(G[B_{l'}]) & \leq \chi(G[N_{G[B_{l'}]}(x)]) + \chi(G[A_{G[B_{l'}]}(x)]) \\ 
&< f(\omega-1, r_{1}, \ldots , r_{t}) + \frac{t}{t+1}M \\
&\leq \frac{M}{t+1} + \frac{t\cdot M}{t+1} = M, 
\end{align*}
which contradicts the fact that $\chi(G[B_{l'}])=M$. This completes the proof of \autoref{cl:pair.of.boxes.nonnbrs.large.chi}.
\end{subproof}

By \autoref{cl:pair.of.boxes.nonnbrs.large.chi}, we have that for every $x\in B_{1}$ there exists a function $f_{x}\colon [2,\ldots, t+2] \rightarrow [t]\setminus I$, such that for every $l\in [2,\ldots, t+2]$ we have $\chi(G[B_{l}\setminus N_{G_{f_{x}(l)}}(x)]) \geq \frac{1}{t+1}M$. Let $g\colon B_{1} \rightarrow [t]^{t+1}$ with $g(x)=(f_{x}(2), f_{x}(3), \ldots , f_{x}(t+2))$. Then $g$ is a $t^{t+1}$-coloring of $G[B_{1}]$. Since $\chi(G[B_{1}])=M> t^{t+1}$, it follows that $g$ is not a proper $t^{t+1}$-coloring of $B_{1}$.

Let $\{u,v\}\subseteq B_{1}$ be an edge of $G$ such that $g(u)=g(v)$. In the $(t+1)$-tuple $g(u)$ at least one element of the set $[t]$ appears more than once. Let $p\in [t]$ be such an element, and let $l,l'\in [2,\ldots,t+2]$ be such that $l\neq l'$ and $p=f_{u}(l)=f_{u}(l')=f_{v}(l)=f_{v}(l')$. Then neither of $\{B_{1}, B_{l}\}$ and $\{B_{1}, B_{l'}\}$ is a $p$-dense pair. Hence, by \autoref{cl:unshrinkable.family}, we have that $p\notin I$, and thus, again by \autoref{cl:unshrinkable.family}, we have that $\{B_{l},B_{l'}\}$ is not a $p$-dense pair.

Let $A_{l}$ be the set $A_{G_{p}}(u)\cap A_{G_{p}}(v)\cap B_{l}$ and let $A_{l'}$ be the set $A_{G_{p}}(u)\cap A_{G_{p}}(v)\cap B_{l'}$.

\begin{claim}
\label{cl:subsets.for.shrinking}
    $\chi(G[A_{l}])\leq c_{p}$ and $\chi(G[A_{l'}])\leq c_{p}$.
\end{claim}

\begin{subproof}[Proof of \autoref{cl:subsets.for.shrinking}]
    We prove that $\chi(G[A_{l}])\leq c_{p}$; the proof that $\chi(G[A_{l'}])\leq c_{p}$ is identical.
    
    Since, by the induction hypothesis, we have that $\prod_{i\in[t]} r_{i}\geq 1$, it follows that there exists $j\in [k]$ such that $r_{j}^{p}\geq 1$. We claim that the graph $G_{p}[A_{l}]$ is $(H_{j}+(r_{j}^{p}-1)K_{2})$-free. Indeed, otherwise $G_{p}[A_{l}(v)\cup\{u,v\}]$ contains $H_{j}+r_{j}^{p}K_{2}$, which contradicts the fact that $G_{p}$ is $\mathcal{H}^{p}$-free. Hence, by the induction hypothesis, we have that $\chi(G[A_{l}]) \leq f(\omega, r_{1}, \ldots , r_{p} - 1, \ldots, r_{t})= c_{p}$. This completes the proof of \autoref{cl:subsets.for.shrinking}.
\end{subproof}

We denote by $\tilde{B_{l}}$ the set $B_{l}\setminus (N_{G_{p}}(u) \cup A_{l})$, and by $\tilde{B_{l'}}$ the set $B_{l'}\setminus (N_{G_{p}}(v) \cup A_{l'})$. Observe that in the graph $G_{p}$ we have that $u$ is complete to $\tilde{B_{l'}}$ and $v$ is complete to $\tilde{B_{l}}$.

\begin{figure}
    \centering
    \includegraphics[page=1]{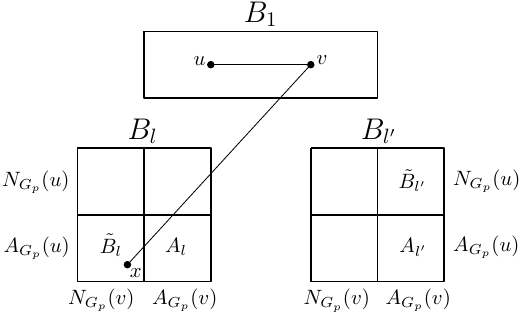}
    \caption{An illustration of the situation in the proof of \autoref{cl:good.pair}. By the definition of $\tilde{B_{l'}}$ we have that $v$ is anticomplete to $\tilde{B_{l'}}$. Thus both $v$ and $x$ are anticomplete to $\tilde{B_{l'}}\setminus N_{G_{p}}(x)$. The disjoint union of an induced  $H_{j}+(r_{j}^{p}-1)K_{2}$ in $\tilde{B_{l'}}\setminus N_{G_{p}}(x)$ with the edge $vx$ would result in an induced $H_{j}+r_{j}^{p}K_{2}$ in $G_{p}$.}
\end{figure}

\begin{claim}
\label{cl:good.pair}
$\{\tilde{B_{l}}, \tilde{B_{l'}}\}$ is a $p$-dense pair in $G$.
\end{claim}

\begin{subproof}[Proof of \autoref{cl:good.pair}]
We prove that for every $x\in \tilde{B_{l}}$ we have $\chi(G[\tilde{B_{l'}}\setminus N_{G_{p}}(x)])\leq c_{p}$. The fact that for every $y \in \tilde{B_{l'}}$ we have $\chi(G[\tilde{B_{l}}\setminus N_{G_{p}}(y)])\leq c_{p}$ follows by symmetry.

Let $x\in \tilde{B_{l}}$. Since, by the induction hypothesis, we have that $r_{p}\geq 1$, it follows that there exists $j\in [k]$ such that $r_{j}^{p}\geq 1$. We claim that the graph $G_{p}[\tilde{B_{l'}}\setminus N_{G_{p}}(x)]$ is $(H_{j}+(r_{j}^{p}-1)K_{2})$-free. Indeed, otherwise $G_{p}[A_{G_{p}}(v)\cup\{x,v\}]$ contains $H_{j}+r_{j}^{p}K_{2}$, which contradicts the fact that $G_{p}$ is $\mathcal{H}^{p}$-free. Hence, by the induction hypothesis, we have that $\chi(G[\tilde{B_{l'}}\setminus N_{G_{p}}(x)]) \leq f(\omega, r_{1}, \ldots ,r_{p} - 1, \ldots, r_{t})= c_{p}$. This completes the proof of \autoref{cl:good.pair}.
\end{subproof}

\begin{claim}
\label{cl:we.can.shrink}
    $\chi(G[\tilde{B_{l}}])\geq \frac{1}{t+2}\chi(G[B_{l}])$ and  $\chi(G[\tilde{B_{l'}}])\geq \frac{1}{t+2}\chi(G[B_{l'}])$.
\end{claim}

\begin{subproof}[Proof of \autoref{cl:we.can.shrink}]
    We prove that $\chi(G[\tilde{B_{l}}])\geq \frac{1}{t+2}\chi(G[B_{l}])$. The proof that $\chi(G[\tilde{B_{l'}}])\geq \frac{1}{t+2}\chi(G[B_{l'}])$ is identical.
    
    By the definition of $g$ and the choice of $l$ we have that $\chi(G[B_{l}\cap A_{G_{p}}(u)])\geq \frac{1}{t+1}M = \frac{1}{t+1}\chi(G[B_{l}])$. By \autoref{cl:subsets.for.shrinking}, we have that $\chi(G[A_{G_{p}}(u)\cap A_{G_{p}}(v)\cap B_{l}])\leq c_{p}$. It follows that: $$\chi(G[\tilde{B_{l}}])\geq \frac{1}{t+1}\chi(G[B_{l}]) - c_{p}\geq \frac{1}{t+1}\chi(G[B_{l}]) - C = \frac{M}{t+1} - C.$$ We also have: 
    \begin{align*}
        M &\geq C(t+1)(t+2) \\
    (t+2)M - (t+1)M &\geq C(t+1)(t+2) \\ 
        \frac{M}{t+1} - \frac{M}{t+2} &\geq C  \\
        \frac{M}{t+1} - C &\geq \frac{M}{t+2}.
    \end{align*}

    Since $\chi(G[B_{l}])=M$ we have that $\frac{M}{t+2} = \frac{1}{t+2}\chi(G[B_{l}])$ and thus by the above we have that $\chi(G[\tilde{B_{l}}])\geq \frac{1}{t+2}\chi(G[B_{l}])$, as desired. This completes the proof of \autoref{cl:we.can.shrink}.
\end{subproof}

By \autoref{cl:good.pair} and \autoref{cl:we.can.shrink}, it follows that the pair $\{B_{l},B_{l'}\}$ is $p$-shrinkable which contradicts the fact that the family $\mathcal{B}$ is unshrinkable. Hence, $$\chi(G)\leq f(\omega(G), r_{1}, \ldots , r_{t}).$$ 

Let $h:\mathbb{N} \to \mathbb{N}$, be defined as follows:

$$h(n)=f(n, \sum_{j\in[k]} r^{1}_{j}, \ldots , \sum_{j\in[k]} r^{t}_{j}).$$

Then we proved that $\chi(G)\leq h(\omega(G))$. Hence, $h$ is a $\chi$-bounding function for the class $\gcap_{i\in[t]}\{\mathcal{H}^{i}\text{-free graphs}\}$. This completes the proof of \autoref{thm:finite.chi.guarding.class.add.K2.self}.
\end{proof}

\section*{Acknowledgements}

We are thankful to Rimma H\"am\"al\"ainen for her involvement in the initial stages of this work and for many fruitful discussions.

\printbibliography

\end{document}